\documentclass[reqno,11pt]{amsart}
\usepackage{amsmath,amsxtra,amsbsy,enumerate, latexsym, amsfonts, amssymb, amsthm, amscd, stmaryrd}
\usepackage{hyperref}
\usepackage{graphics,epsf,psfrag,accents}
\usepackage{bbm, dsfont}
\usepackage{mathtools}
\usepackage{xcolor}
\usepackage{float}
\usepackage{tikz}
\usepackage[headings]{fullpage}
\usepackage{colonequals}
\usepackage{microtype}
\numberwithin{equation}{section}

\newcommand{\bbN}{{\ensuremath{\mathbb N}} }

\newcommand{\bbP}{{\ensuremath{\mathbb P}} }

\newcommand{\bbR}{{\ensuremath{\mathbb R}} }

\newcommand{\bbZ}{{\ensuremath{\mathbb Z}} }

\newcommand{\IBN}{{\ensuremath{\mathrm {IBN}}} }

\newcommand{\Igr}{{\ensuremath{\underline  {\rm I}_{gr}}} }

\newcommand{\ga}{\alpha}
\newcommand{\gb}{\beta}
\newcommand{\gga}{\gamma}            

\newcommand{\gep}{\varepsilon}       

\newcommand{\gl}{\lambda}
\newcommand{\gL}{\Lambda}

\newcommand{\cG}{{\ensuremath{\mathcal G}} }
\newcommand{\cO}{{\ensuremath{\mathcal O}} }

\newcommand{\cA}{{\ensuremath{\mathcal A}} }

\newcommand{\cC}{{\ensuremath{\mathcal C}} }

\newcommand{\cT}{{\ensuremath{\mathcal T}} }

\newcommand{\bP}{{\ensuremath{\mathbf P}} }

\newcommand{\cR}{{\ensuremath{\mathcal R}} }

\newcommand{\supp}{\ensuremath{\mathrm{supp} }}
\newtheorem{theorema}{Theorem}

\newcommand{\ind}{\mathbf{1}}

\newcommand{\lint}{\llbracket}
\newcommand{\rint}{\rrbracket}

\definecolor{darkred}{rgb}{0.7,0.1,0.1}

\definecolor{darkgreen}{rgb}{0.1,0.7,0.1}

\newtheorem{theorem}{Theorem}[section]
\newtheorem{lemma}[theorem]{Lemma}
\newtheorem{proposition}[theorem]{Proposition}
\newtheorem{cor}[theorem]{Corollary}
\newtheorem{rem}[theorem]{Remark}
\newtheorem{definition}[theorem]{Definition}
\newtheorem{Problem}[theorem]{Problem}
\newtheorem{Example}[theorem]{Example}
\newtheorem{question}[theorem]{Question}
\newtheorem{fact}[theorem]{Fact}

\newcommand{\RN}[1]{%
  \textup{\uppercase\expandafter{\romannumeral#1}}%
}

\makeatletter
\def\captionfont@{\footnotesize}
\def\captionheadfont@{\scshape}

\long\def\@makecaption#1#2{%
  \vspace{2mm}
  \setbox\@tempboxa\vbox{\color@setgroup
    \advance\hsize-6pc\noindent
    \captionfont@\captionheadfont@#1\@xp\@ifnotempty\@xp
        {\@cdr#2\@nil}{.\captionfont@\upshape\enspace#2}%
    \unskip\kern-6pc\par
    \global\setbox\@ne\lastbox\color@endgroup}%
  \ifhbox\@ne 
    \setbox\@ne\hbox{\unhbox\@ne\unskip\unskip\unpenalty\unkern}%
  \fi
  \ifdim\wd\@tempboxa=\z@ 
    \setbox\@ne\hbox to\columnwidth{\hss\kern-6pc\box\@ne\hss}%
  \else 
    \setbox\@ne\vbox{\unvbox\@tempboxa\parskip\z@skip
        \noindent\unhbox\@ne\advance\hsize-6pc\par}%
\fi
  \ifnum\@tempcnta<64 
    \addvspace\abovecaptionskip
    \moveright 3pc\box\@ne
  \else 
    \moveright 3pc\box\@ne
    \nobreak
    \vskip\belowcaptionskip
  \fi
\relax
}
\makeatother
\def\writefig#1 #2 #3 {\rlap{\kern #1 truecm
\raise #2 truecm \hbox{#3}}}

\makeatletter
\newsavebox{\@brx}
\newcommand{\llangle}[1][]{\savebox{\@brx}{\(\m@th{#1\langle}\)}%
  \mathopen{\copy\@brx\kern-0.5\wd\@brx\usebox{\@brx}}}
\newcommand{\rrangle}[1][]{\savebox{\@brx}{\(\m@th{#1\rangle}\)}%
  \mathclose{\copy\@brx\kern-0.5\wd\@brx\usebox{\@brx}}}
\makeatother

\title[The branching number of  intermediate growth trees]{The branching number of  intermediate growth trees}

\author[Gideon Amir]{Gideon Amir}
 \address{Gideon Amir \hfill\break
Bar-Ilan University, 5290002, Ramat Gan, Israel.}
\email{: gidi.amir@gmail.com}

\author[Shangjie Yang]{Shangjie Yang}
 \address{Shangjie Yang \hfill\break
Bar-Ilan University, 5290002, Ramat Gan, Israel.}
\email{shangjie.yang@biu.ac.il}


\begin{document}


\begin{abstract}

We introduce an "intermediate branching number"(IBN) which captures the branching of intermediate growth trees, similar in spirit to the well-studied branching number of exponential growth trees.
We show that the IBN is the critical threshold for several random processes on trees, and analyze the IBN on some examples of interest. Our main result is an algorithm to find spherically symmetric trees with large IBN inside some permutation wreath products. We demonstrate the usefulness of these trees to the study of intermediate growth groups by using them to get the first tight bounds for the firefighter problem on some inetrmediate growth groups.

\end{abstract}

\maketitle

\section{Introduction}\label{sec:intro}

Given an infinite rooted tree, there are several ways to measure the amount of branching and structure of the tree. One obvious such measure is the volume growth (of the balls around the root). However, the growth rate does not capture many of the properties of the tree, and in particular does not capture the behavior of random walks on the tree. e.g. one can have an exponential growth tree on which simple random walk is recurrent. This motivates finding other measures. One such quantity, which proved to be very successful, is the \emph{branching number} (cf. \cite[Chapter 3]{Lyons2016trees}).  The branching number was linked to the critical values of different processes on trees such as transience/recurrence of biased random walk on trees,  percolation thresholds, and firefighting. One limitation of the branching number is that it is always $1$ for subexponential growth trees, thus it does not convey any information on them.
A step in generalizing the branching number to smaller trees was done in \cite{Collevecchio2020branching} where an analogous branching-ruin number on polynomial growth trees was introduced and studied. While it managed to capture threshold properties for several random processes on polynomial growth trees (e.g. the one-reinforced random walk), the branching-ruin number is infinite for superpolynomial growth trees. Thus there is a gap left not covered by either of these branching numbers.

In this paper we introduce a new branching number, which we call the \emph{intermediate branching number} (IBN), to capture the branching structure of intermediate growth trees (or more precisely  trees with stretched exponential growth). Its exact definition is given in section \ref{sub:IBN}. One of the motivation for studying intermediate growth trees comes from their connections to the Cayley graphs of intermediate growth groups. Even though it has been about 40 years since the introduction of the first group of intermediate growth by Grigorchuk \cite{Grigorchuk1983Milnor}, the geometry of these groups and their Cayley graphs is not well understood. Despite of lots of researches and breakthroughs in the area, many processes, which are well understood on exponential and on polynomial growth groups, remain mostly mysterious on intermediate growth groups. Thus we believe that any additional tools to study the Cayley graphs of such groups are of interest.

We begin by giving  some basic results and examples concerning intermediate growth trees and their IBN,  and analyse the IBN for the lexicographical minimal spanning tree of Nathanson's intermediate growth semigroup. We then go on to show that the IBN captures the threshold for several random processes on trees, such as random walks with  (random) conductances and percolation, similar to the connection between the branching number and the branching-ruin number on exponential and on polynomial growth trees respectively.

Our main result shows how to find "large symmetric" trees inside permutation wreath products. In other words, we prove the existence of spherically symmetric trees with a large intermediate branching number. More precisely we give a lower bound on the intermediate branching number of these trees in terms of the inverted orbit growth of the action of these groups (See section \ref{sec:tree-perwproduct} for definitions relating to permutation wreath products).
 Permutation wreath products play an important role in the construction of families of groups of intermediate growth \cite{Bartholdi2012permutational,  Bartholdi2014growth}, and in particular in the construction of intermediate growth groups with given stretched exponential growth. Our result allows us to construct spherically symmetric trees in the Cayley graphs of a large family of these groups with IBN equal to their stretched exponential growth rate. We use the trees to show that the critical threshold for the firefighting problem on these intermediate growth groups is equal to their growth rate. This constitutes the first tight threshold for the firefighting problem on intermediate growth groups.



\subsection*{Organization of the paper}
Section \ref{sec:modres} introduces our setup, notation and the definition of the intermediate branching number. Section \ref{sec:IBNmany} studies the relations between the intermediate branching number and the thresholds of certain stochastic processes on trees, such as random walks with (random) conductance and percolation. In section \ref{sec:CayleyTree} we discuss intermediate growth trees in Cayley graphs, with a focus on the lexicographic minimal spanning tree of Nathanson's semigroup example. Section \ref{sec:tree-perwproduct} is devoted to the construction of large spherically symmetric trees inside permutation wreath products. It also includes the needed background regarding permutation wreath products and inverted orbits. In section \ref{sec:fireTreeGroup} we discuss the firefighting problem in general and on intermediate growth trees and groups, and use the results from previous  sections to get the tight bounds mentioned in the introduction. Section \ref{sec:noperiodtree} is devoted to the proof of Proposition \ref{th:intersymsub}. Some of the proofs, namely the ones which we believe follow along similar lines to the exponential growth case, are given in the appendix. Appendices \ref{sec:rwphase}, \ref{sec:percolate}, \ref{sec:rwrandcond}, \ref{sec:fireIBN} are devoted to Theorems \ref{th:IBNrw}, \ref{th:percolate}, \ref{th:RWRCIBN} and \ref{th:fireintermediate}   respectively.


\section{Model and results}\label{sec:modres}

\subsection{Notation and setup}

Given a rooted graph $G=(V,E,\varrho)$ we denote by $\vert v \vert$ the distance between a vertex $v$ and the root $\rho$. When $G$ is the Cayley graph of some group we will identify $\rho$ with the identity element.

Let $\cT=(V,E, \varrho)$ be a locally finite, infinite rooted tree. For an edge $e \in E$, let $e^+$ and $e^-$ be the two end points of $e$ with $\vert e^+\vert= \vert e^-\vert+1$. That is to say, $e^+$ is  a child of $e^-$. For $\mu, \nu \in V$, let $[\mu, \nu]$ be the unique self-avoiding path connecting $\mu$ and $\nu$.
If $\mu$ is on the path $[\varrho, \nu]$, we say $\mu \le \nu$, and $\mu< \nu$  if $\mu \neq \nu$. Similarly, we define such a relation between two edges $e_1,e_2 \in E$. In addition, let $e_1 \wedge e_2$ be the common ancestor of the deepest generation of $e_1$ and $e_2$, \textit{i.e.}
\begin{equation*}
e_1 \wedge e_2 \colonequals \sup \left\{e \in E: \ e\le e_1, e\le e_2 \right\}.
\end{equation*}
Let $B_n \colonequals \left\{x \in V: \ \vert x \vert \le n \right\}$ be the set of those vertices at distance at most $n$ (in the graph distance) to $\varrho$. Moreover, set $E_n \colonequals \{ x \in V : \ \vert x\vert=n\}$ to be the set of vertices at distance exactly  $n$ to the root $ \varrho$.

We let $\bbN_0 \colonequals \bbZ \cap [0, \infty)$, and for $a,b>0$ let $\lint a, b\rint \colonequals [a,b] \cap \bbZ$. We let $\#$ denote the cardinality of a set.

\subsection{Intermediate  growth trees}
Given a rooted tree $\cT$, we define its lower and upper intermediate (stretched exponential) growth rates to be
\begin{equation}
\begin{gathered}
\underline{ \rm I}_{gr}(T):=\liminf_{n \to \infty} \log \log  \#E_n/\log n \, ; \\ \overline{\rm I}_{gr}(T):=\limsup_{n \to \infty} \log \log \# E_n/\log n.
\end{gathered}
\end{equation}
We say that a tree $\cT$ is of intermediate (stretched exponential) growth if $0<\underline{\rm I}_{gr}
(\cT)\leq \overline{\rm I}_{gr}(\cT)<1$. We say $\cT$ has intermediate growth ${\rm I}_{gr}(\cT)=\alpha$, if $\underline{\rm I}_{gr}
(\cT)\leq \overline{\rm I}_{gr}(\cT)=\ga$.
Note that $\ga$ only determines the first order of the growth rate, but there may be lower order corrections.


\medskip

 We first give what is, perhaps, the simplest example of an intermediate growth tree. Take a sequence of numbers $1,2,1,1,2,1,1,1,2,1,1,1,1,2,...$, denoted by $(a_n)_{n \in \bbN}$,  where the amount of $1$s increases by one between any two consecutive $2$s.
In other words, the sequence $(a_n)_{n \in \bbN}$ is defined as follows: for $n \ge 1$
\begin{equation}
a_n \colonequals
\begin{cases} 2 &\text{ if } n = k+\frac{(1+k)k}{2} \text{ for some } k \in \bbN,\\
1 &\text{ otherwise. }
\end{cases}
\end{equation}
We construct a tree $\cT$ such that the root $\varrho$ has only one child and every vertex $x$ with $\vert x \vert=n \ge 1$ has $a_n$ children.  
Since $$\# E_n=2^{\sum_{k=1}^{n-1} (a_k-1)},$$
we have ${\rm I}_{gr}(\cT)=1/2$.
Furthermore,
a tree is spherically symmetric if every vertex at distance $n$ to the root has the same number of children (which can depend on $n$). The example above is spherically symmetric.

\subsection{The intermediate branching number.}\label{sub:IBN}
Like in the exponential growth case, the intermediate growth rate of a tree fails to capture many of its essential properties, in particular it fails to characterize the behaviour of many random processes on the tree. In particular, it is possible to construct an example of an intermediate growth tree on which simple random walk is recurrent. (See e.g. Example \ref{example} below, which  is a variant of the famous $3-1$ tree).
Since $\Igr(\cT)$ only captures the volume growth of a tree, to characterize its structure and complexity we mimic \cite[Eq. (3.4)]{Lyons2016trees} to define the intermediate branching number.
We say that $\pi \subset E$ is a cutset if any infinite self-avoiding path connecting $\varrho$ to infinity has one unique edge in $\pi$.

\begin{definition}
 The intermediate branching number of a tree $\cT$ is
\begin{equation}\label{def:IBN}
\IBN(\cT)\colonequals \sup \left\{ \gl>0: \ \inf_{\pi \in \Pi}\sum_{e \in \pi} \exp\left(- \vert e \vert^{\gl} \right) >0 \right\}
\end{equation}
where $\Pi$ is the set of all cutsets.
\end{definition}

Note that  since $(E_n)_{n \in \bbN} \subset \Pi$, we have
\begin{equation}\label{cutgrow}
 \IBN(\cT) \le  \underline{\rm I}_{gr}(\cT).
\end{equation}
 Under additional symmetry one may hope $\IBN(\cT) = \underline{\rm I}_{gr}(\cT)$.\

\begin{fact}
For a spherically symmetric tree $\cT$, 
we have
\begin{equation}\label{eq:spherical}
\IBN(\cT)= \Igr(\cT).
\end{equation}
\end{fact}
\begin{proof}
This follows from  the max-flow min-cut theorem \cite[Theorem 3.1]{Lyons2016trees} 
  applied to the flow from the root to infinity that splits equally among the children at any vertex.
\end{proof}
Note that without the spherical symmetry requirement the equality in \eqref{eq:spherical} may fail, as can be seen by the following example:

\begin{Example}\label{example}
We take the $3-1$ tree \cite[Example 1.2]{Lyons2016trees} and
replace each edge at distance $n$  by a path of length $n$ to obtain a tree satisfying  $\Igr(\cT)=1/2$ and $\IBN(\cT)=0$.
\end{Example}

\section{The IBN as the threshold for stochastic processes on trees}
\label{sec:IBNmany}

The (exponential) branching number and the branching-ruin number are known to capture the threshold of several stochastic processes on trees. Not all these properties carry over to intermediate growth trees, but some do. In the following subsections, we demonstrate that the intermediate branching number also serves as the critical threshold for several stochastic processes.

Before going on to these results, let us give an interesting example where we do not have a satisfactory intermediate growth analog. It was shown by Lyons (cf. \cite[Theorem 3.5]{Lyons2016trees}) that for the homesick random walk (that gets a fixed bias towards the root) the branching number serves as the critical threshold for recurrence vs. transience. Collevecchio, Kious and Sidoravicius \cite[Theorem 2.1]{Collevecchio2020branching} proved that the branching-ruin number is the transience vs. recurrence threshold for the once-reinforced random walk on polynomial growth trees. This raises the following question:
\begin{Problem}
Find a "natural" random walk between the once-reinforced and the home sick random walks such that the IBN will serve as the transience vs. recurrence threshold for that walk on intermediate growth trees.
\end{Problem}

\subsection{The recurrence/transience phase transition of random walks}
In this subsection, we are concerned with the  recurrence/transience phase transition of the random walk on an intermediate growth tree $\cT=(V,E)$ with conductances given as $ c(e)=\exp(- \vert e\vert^{\gl})$ for all $e \in E$.
Similar results linking the branching number of exponential growth trees and the phase transition of biased random walks (taking $C(e)=\lambda^{-|e|}$) can be found in \cite{Lyons1990}, \cite[Theorem 3.5]{Lyons2016trees}.

We define a random walk $\textbf{X}=(X_n)_{n \ge 0}$ with $X_0=\varrho$ and for $n \ge 1$
\begin{equation}\label{def:rw}
\bbP \left[ X_n=\nu \ \vert \ X_{n-1}=\mu \right]\colonequals \frac{c([\mu, \nu]) \ind_{\{ \mu \sim \nu\}}}{\sum_{\nu': \nu' \sim \mu} c([\mu, \nu'])}
\end{equation}
where  $\mu \sim \nu$ denotes that $\mu$ and $\nu$ are connected by an edge and we allow $\mu < \nu'$ or $\mu> \nu'$.
In other words,  when $X_{n-1}=\mu$, $X_n$ picks a neighbor among all the neighbors of $\mu$ with probability proportional to the conductance. We link the phase transition of the random walk with the intermediate branching number in the following theorem.

\begin{theorem}\label{th:IBNrw}
Let $\textbf{X}=(X_n)_{n \ge 0}$ be the random walk defined above  associated with conductance $c(e)=\exp \left(- \vert e\vert^{\gl} \right)$. If $\gl>\IBN(\cT)$, $\textbf{X}$ is recurrent. If $\gl \in (0, \IBN(\cT))$, $\textbf{X}$ is transient.
\end{theorem}
The proof for  Theorem \ref{th:IBNrw} is adapted from \cite{Lyons2016trees}.  For the sake of completeness, we include a proof in Appendix \ref{sec:rwphase}.


\subsection{The intermediate branching number and percolation}
Given an intermediate growth tree, we consider the following percolation. We first fix $\gl \in (0,1)$.  Each edge is open independently, and for $e \in E$,
\begin{equation}\label{perc:prob}
\bbP \left[ e \text{ is open}  \right]=\exp \left(- \vert e\vert^{-1+\gl} \right).
\end{equation}
We define the critical parameter about percolation as
\begin{equation}\label{perc:threshold}
\Theta(\cT) \colonequals \sup \left\{ \gl \ge 0: \  \bbP \left[ \varrho  \leftrightarrow \infty \right]>0 \right\}.
\end{equation}

\begin{theorem}\label{th:percolate}
 In the setting of \eqref{perc:prob} and \eqref{perc:threshold},  we have
\begin{equation}
\Theta(\cT)=\IBN(\cT).
\end{equation}
\end{theorem}

 Theorem \ref{th:percolate} is analogous to \cite{Lyons1990}, \cite[Theorem 5.15]{Lyons2016trees} linking the branching number and the percolation threshold in exponential growth trees. We include a proof in Appendix \ref{sec:percolate} for the sake of completeness.


\subsection{Random walks with random conductance}
In \cite{Collevecchio2019branching} the authors study a random walk in heavy-tailed random conductances on polynomial growth trees. They prove a phase transition between recurrence and transience depending on the tail of the law of the conductances, and show that the branching-ruin number is the critical threshold for this transition.
Our aim in this section is to show a similar result for intermediate growth trees and show that the critical threshold for recurrence vs transience is the intermediate branching number.
We start by describing the random conductance model.

Let $(C_e)_{e \in E}$ be a sequence of i.i.d. random variables with common law $\bbP$ supported on $(0, \infty)$.
We assume there exists $\gl \in (0,1)$ such that for all $t \ge 1$
\begin{equation}\label{asm:conduct}
\bbP \left[  C_1 < \exp \left( - t^{\gl}\right)  \right]=L(t) \cdot t^{\gl-1}
\end{equation}
where $L: [1,\infty) \mapsto \bbR_+$ is a slowly varying function, \textit{ i.e.} for all $a>0$
$$ \lim_{t \to \infty}  \frac{L(at)}{L(t)}=1.$$
In other words, the random conductance $C_e$ has a heavy tail at $0$. Given a realization of $(C_e)_{e \in E}$, we define the random walk $\textbf{X}=(X_n)_{n  \ge 0}$ with  conductance $(C_e)_{e \in E}$ exactly the same as \eqref{def:rw}. From now on, we abbreviate $\textbf{X}=(X_n)_{n  \ge 0}$, the random walk on random conductance, as RWRC.

\begin{theorem}\label{th:RWRCIBN}
If $\gl> \ga=\IBN(\cT)$, then RWRC 
is recurrent almost surely. If $\gl< \ga=\IBN(\cT)$, then RWRC is transient a.s. 
\end{theorem}

 We follow the blueprint in \cite{Collevecchio2019branching} to prove Theorem \ref{th:RWRCIBN}, which is presented in Appendix \ref{sec:rwrandcond}.

\section{Intermediate growth trees and Cayley graphs}\label{sec:CayleyTree}

In the next two sections we study the intermediate branching number of several families of trees inside Cayley graphs of intermediate growth groups and semigroups.
The motivation is twofold. On the one hand, such trees would provide interesting examples and may inherit some properties from the Cayley graphs. On the other hand, we hope that finding good trees inside the Cayley graphs would provide tools to  understand better the geometry of these graphs and their underlying groups.

\subsection{Lexicographical minimal spanning trees}
One natural family of trees of intermediate growth are the \textit{lexicographical minimal spanning trees}(LMST) of Cayley graphs of intermediate growth. 
Such trees were studied for exponential growth groups (see \cite{Lyons1995walks,Lyons2016trees}). Since the vertices at distance $n$ from the root in the LMST is simply the set of group elements of length $n$, the growth rate of the LMST is simply the growth rate of the group (w.r.t. the same generating set).
It is interesting to note that unlike the exponential growth case, there are examples groups of "oscillating" intermediate growth (e.g. in \cite[Theorem 2.3]{Kassabov2013osilating},   \cite{Bartholdi2014growth},    
 \cite{Brieussel2014Growth}) for which the  LMST satisfies $\underline{\rm I}_{gr}(\cT)< \overline{\rm I}_{gr}(\cT)$. In addition,
one important structural properties of the lexicographical minimal spanning tree is that it is  \emph{sub-periodic}.

 \begin{definition}
For a given tree $\cT=(V,E)$ and $x\in V$, let $\cT_x$ be the subtree consisting of $x$ and all descendants of $x$.
We say that  $\cT$ is $M$-sub-periodic if for all $x \in V$ there exists an injective map $f: \cT_x \mapsto \cT$ satisfying $\vert f(x) \vert \le M$ and keeping the adjacency of the vertices. If there exists one such $M \in \bbN_0$, we say that $\cT$ is sub-periodic.
\end{definition}
In fact, it is straightforward to see that any LMST is $0$-sub-periodic.

\begin{rem}
While it is easy to come up with examples of sub-periodic spherically symmetric exponential growth trees (e.g. binary tree), it turns out one cannot have the same for intermediate growth trees. That is the following proposition, whose proof is given in Section \ref{sec:noperiodtree}.
\begin{proposition}\label{th:intersymsub}
There is no spherically symmetric,   sub-periodic and subexponential growth tree, \textit{i.e.}
\begin{equation}\label{assum:subexp}
 \lim_{n \to \infty} \frac{\log \# E_n}{n}=0,
\end{equation}
and satisfying $\lim_{n \to \infty} \# E_n=\infty.$
\end{proposition}
\end{rem}

For sub-periodic exponential growth trees,  a theorem of Furstenberg (\cite{Furstenberg67},\cite[Theorem 3.8]{Lyons2016trees}) shows that the branching number is equal to the growth rate.
This implies that for exponential growth groups, the LMST has branching number equal to the group's growth rate.
However, this line of proof no longer works for intermediate growth trees, as was proved recently by Tang:
\begin{theorem}[{\cite[Theorem 1.2]{Tang2022branching}}]\label{th:tang}
There exists a sub-periodic tree of intermediate growth $T$ for which $\IBN(\cT)<\underline{\rm I}_{gr}(\cT)$.
\end{theorem}
Note that Tang's exmaple is not an LMST, but rather some amalgamation of "shifts" of the "3-1" tree.
Furthermore,
in general we do not know how to calculate the IBN for LMST of intermediate growth groups. In particular, we ask the following question:
\begin{question}
Is it true that $\IBN(\cT)=\underline{\rm I}_{gr}(T)$ for any ${\rm LMST}$\ of an intermediate growth group? What about the ${\rm LMST}$ of Grigorchuk's group?
\end{question}

\subsection{Nathanson's semigroup example}

Since Cayley graphs of intermediate growth groups (and their LMST) are very hard to understand, we analyze in this section the LMST of Nathanson's intermediate growth semigroup.
%
%
%
 In \cite{Nathanson1999intermediate}, Nathanson provided an interesting example of a semigroup of intermediate growth: a semigroup of $2 \times 2 $ matrices generated by the set $A=\{a,b \}$ w.r.t. multiplication where
\begin{equation}\label{gen:semigr}
a \colonequals \begin{pmatrix}
1 & 1\\
0 & 1
\end{pmatrix}
\quad
\text{ and }
\quad
b \colonequals \begin{pmatrix}
1 & 0\\
1 & 0
\end{pmatrix}.
\end{equation}
Its volume growth is given in the following theorem.
 \begin{theorema}  {\cite[Theorem 1]{Nathanson1999intermediate}}\label{th:Nathanson}
  Let $S$ be the semigroup generated by $\{ a,b\}$ defined in \eqref{gen:semigr}.
Then there exists a constant $c>0$ such that for all $n$ sufficiently large,
\begin{equation}
2^{c \sqrt{n}/\log n}  \le  \# B_n \le 2 n^{2 \sqrt{n}+2}.
\end{equation}
 \end{theorema}
We define a lexicographical order on $S$ by setting $b<a$ to obtain its lexicographic minimal spanning tree as in Figure \ref{fig:semi}.
There are three types of words in $E_n$: (1) $a^n$; (2) $a^i b a^{j}$  with  $i,j \in \bbN_0 \text{ and } i+j=n-1$; and (3)
\begin{equation}\label{fig:princ}
 a^i \underbrace{ b a^{p_1-1}  \cdots   ba^{p_1-1}}_{r_1 \text{ copies }}  \cdots \underbrace{b a^{p_k-1}  \cdots ba^{p_k-1}}_{r_k \text{ copies }} b a^j
\end{equation}
  where
  $i+j+1+\sum_{\ell=1}^k r_{\ell}p_{\ell} =n$  and  $p_1<p_2< \cdots <p_k \text{ are primes}$. All these three types of words are written in the lexicographic minimal sense. For two words $g \in E_n$, $h \in E_{n+1}$ written in the forms stated above, there is an edge linking $g$ with $h$ if $h=ga$ or $h=gb$.
 The figures (M) and (R) in Figure \ref{fig:semi} explain the rule in \eqref{fig:princ}.
The following lemma concerns the size of sphere at distance $n$.

 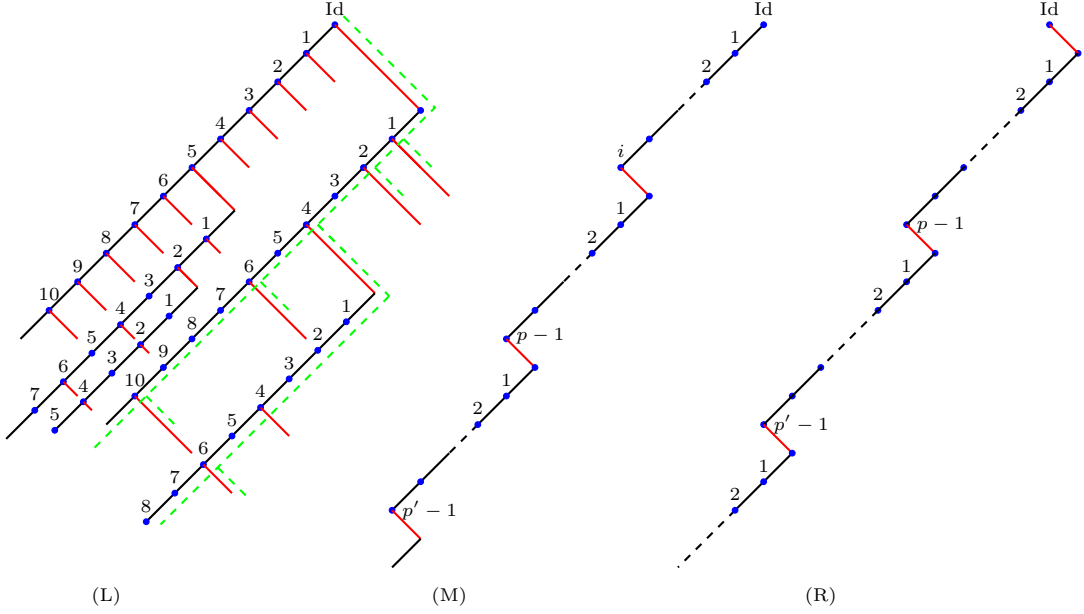
\begin{figure}[h]
 \centering
   \begin{tikzpicture}[scale=.38,font=\tiny]
   \node[above] at (0,0) {$\mathrm{Id}$};
    \draw[thick] (0,0) -- (-11,-11);
     \foreach \x in {0,-1,-2,...,-10} {\draw[fill,blue] (\x,\x) circle [radius=0.1];}
     \foreach \x in {1,2,...,10} { \node[above] at (-\x,-\x) {\x}; }
          \foreach \x in {1,2,...,10} { \draw[thick,red] (-\x,-\x) -- (-\x+1,-\x-1); }

   \draw[thick,red] (-5,-5) -- (-5+1.5,-5-1.5);
    \draw[thick]  (-5+1.5,-5-1.5)-- (-5+1.5-8,-5-1.5-8);
   \foreach \x in {1,2,...,7} { \node[above] at (-5+1.5-\x,-5-1.5-\x) {\x}; }
   \foreach \x in {1,2,...,7} { \draw[fill,blue](-5+1.5-\x,-5-1.5-\x)  circle [radius=0.1]; }
   \foreach \x in {1,2,4,6} { \draw[red,thick](-5+1.5-\x,-5-1.5-\x)--(-5+1.5-\x+0.5,-5-1.5-\x-0.5); }
   \draw[red,thick](-5+1.5-2,-5-1.5-2)--(-5+1.5-2+0.7,-5-1.5-2-0.7);
   \draw[thick](-5+1.5-2+0.7,-5-1.5-2-0.7)--(-5+1.5-2+0.7-5,-5-1.5-2-0.7-5);
    \foreach \x in {1,2,...,5} { \node[above] at (-5+1.5-2+0.7-\x,-5-1.5-2-0.7-\x) {\x}; }
       \foreach \x in {1,2,...,5} { \draw[fill,blue] (-5+1.5-2+0.7-\x,-5-1.5-2-0.7-\x) circle [radius=0.1]; }
    \foreach \x in {2,4} { \draw[red,thick] (-5+1.5-2+0.7-\x,-5-1.5-2-0.7-\x)--(-5+1.5-2+0.7-\x+0.3,-5-1.5-2-0.7-\x-0.3); }

     \draw[thick,red] (0,0) -- (1+1+1,-1-1-1);
     \draw[thick]  (1+1+1,-1-1-1)--(1-10+1-1+1,-1-10-1-1-1);
   \foreach \x in {0,-1,-2,...,-10} {\draw[fill,blue] (1+1+1+\x,-1+\x-1-1) circle [radius=0.1];}
        \foreach \x in {1,2,...,10} { \node[above] at (1+2-\x,-\x-2-1) {\x}; }

   \draw[thick,red] (1+1,-3-1) -- (3,-5);
      \draw[thick,red] (1+1,-3-1) -- (3+1,-5-1);
       \draw[thick,red] (0+1,-4-1) -- (2+1,-6-1);
       \draw[thick,red] (0-2+1,-4-2-1) -- (2-2+1+0.4,-6-2-1-0.4);
       \draw[thick,red] (0-4+1,-4-4-1) -- (2-4+1,-6-4-1);
       \draw[thick,red] (0-8+1,-4-8-1) -- (2-8+1,-6-8-1);

   \draw[thick](2-2+1+0.4,-6-2-1-0.4)--(2-2+1+0.4-8,-6-2-1-0.4-8);
   \foreach \x in {1,2,...,8} { \node[above] at (2-2+1+0.4-\x,-6-2-1-0.4-\x) {\x}; }
   \foreach \x in {1,2,...,8} { \draw[fill,blue]  (2-2+1+0.4-\x,-6-2-1-0.4-\x)circle [radius=0.1]; }
 \draw[thick,red](2-2+1+0.4-4,-6-2-1-0.4-4)--(2-2+1+0.4-3,-6-2-1-0.4-5);
     \draw[thick,red](2-2+1+0.4-4-2,-6-2-1-0.4-4-2)--(2-2+1+0.4-3-2,-6-2-1-0.4-5-2);

    \node at (-8,-20){(L)};
    \node at (4,-20){(M)};
    \node at (17,-20){(R)};

\draw[green,thick,dashed](0+0.3,0+0.3)--(2+1+0.3+0.2,-2-1+0.3-0.2)--(2+1+0.3-7+0.2-5,-2-1+0.3-7-0.2-5);
 \foreach \x in {1,2,4,6,10} { \draw[green,thick,dashed](2+1+0.3+0.1-\x,-2-1+0.3-0.3-\x)--(2+1+0.3+0.1-\x+1,-2-1+0.3-0.3-\x-1); }
\draw[green,thick,dashed](2+1+0.3+0.1-4,-2-1+0.3-0.3-4)--(2+1+0.3+0.1-4+2+0.5,-2-1+0.3-0.3-4-2-0.5);
  \draw[green,thick,dashed](2+1+0.3+0.1-4+2+0.5,-2-1+0.3-0.3-4-2-0.5)--(2+1+0.3+0.1-4+2+0.5-8,-2-1+0.3-0.3-4-2-0.5-8);
  \draw[green,dashed,thick](2+1+0.3+0.1-4+2+0.5-6,-2-1+0.3-0.3-4-2-0.5-6)--(2+1+0.3+0.1-4+2+0.5-6+1,-2-1+0.3-0.3-4-2-0.5-6-1);

 \draw[thick] (15,0) --(13,-2);
\draw[thick,dashed] (13,-2)--(12,-3);
 \draw[thick]  (11+1,-4+1)-- (5+4+1,-6+1);
 \foreach \x in {0,1,2,4,5} { \draw[fill,blue]  (15-\x,-\x)circle [radius=0.1]; }
\node[above] at(15,0){Id};
\node[above] at(14,-1){$1$}; \node[above] at(13,-2){$2$};\node[above] at(15-6+1,-6+1){$i$};
\draw[thick,red] (15-6+1,-6+1)--(10+1,-7+1);

\draw[thick] (10+1,-7+1)--(10-3+1+1,-7-3+1+1);
\draw[thick,dashed] (10-3+1+1,-7-3+1+1)--(10-5+1+2,-7-5+1+2);
\draw[thick] (10-5+1+1+1,-7-5+1+1+1)--(10-7+1+1+1,-7-7+1+1+1);
  \foreach \x in {0,1,2,4,5} { \draw[fill,blue]  (1+10-\x,1-7-\x)circle [radius=0.1]; }
\node[above] at(10,-7){$1$};\node[above] at(10-1,-7-1){$2$}; \node[right] at(10-6+1+1,-7-6+0.1+2){$p-1$};
\draw[red,thick](10-7+1+2,-7-7+1+2)--(10-7+2+2,-7-7+2);
\draw[thick](10-7+2+2,-7-7+2)--(10-7+2,-7-7);
\draw[thick,dashed](10-7+2,-7-7)--(10-7+1,-7-7-1);
\draw[thick](10-7+1,-7-7-1)--(10-7-1,-7-7-3);
  \foreach \x in {0,1,2,4,5} { \draw[fill,blue]  (7-\x,-12-\x)circle [radius=0.1]; }
\node[above] at(7-1,-12-1){$1$};\node[above] at(7-2,-12-2){$2$}; \node[right] at(7-5,-12-5){$p'-1$};
\draw[red,thick](7-5,-12-5)--(7-5+1,-12-5-1);
\draw[thick](7-5+1,-12-5-1)--(7-5,-12-5-2);


 \node[above] at(25,0){Id};
 \draw[fill,blue]  (25,0)circle [radius=0.1];
    \foreach \x in {0,1,2,4,5,6} { \draw[fill,blue]  (26-\x,-1-\x)circle [radius=0.1]; }
    \foreach \x in {1,2} { \node[above] at (26-\x,-1-\x){\x}; }
\node[right] at (26-6,-1-6){$p-1$};

    \draw[thick,red] (25,0) -- (26,-1);
    \draw[thick]  (26,-1)--(26-2,-1-2);
    \draw[thick,dashed]   (26-2,-1-2)--(26-4,-1-4);
    \draw[thick]  (26-4,-1-4)--(26-6,-1-6);

    \draw[thick,red]  (26-6,-1-6)--(26-5,-1-7);
    \foreach \x in {0,1,2,4,5,6} { \draw[fill,blue]  (26-5-\x,-1-7-\x)circle [radius=0.1]; }
    \foreach \x in {1,2} { \node[above] at (26-5-\x,-1-7-\x){\x}; }
    \node[right] at (26-5-6,-1-7-6){$p'-1$};

   \draw[thick]  (26-5,-1-7)--(26-7,-1-9);
   \draw[thick,dashed]  (26-7,-1-9)--(26-9,-1-11);
   \draw[thick]  (26-9,-1-11)--(26-11,-1-13);

   \draw[thick,red]  (26-11,-1-13)--(26-10,-1-14);

    \draw[thick] (26-10,-1-14)--(26-10-2,-1-14-2);
    \draw[thick,dashed] (26-10-2,-1-14-2)--(26-10-4,-1-14-4);
     \foreach \x in {0,1,2} { \draw[fill,blue]  (26-10-\x,-1-14-\x)circle [radius=0.1]; }
    \foreach \x in {1,2} { \node[above] at (26-10-\x,-1-14-\x){\x}; }

   \end{tikzpicture}

  \caption{ A graphical explanation: a black edge represents $a$ while a red edge represents $b$ with a blue dot standing for a vertex. (L): the general picture of the lexicographic minimal spanning tree (note that not all the branches are shown.), and the green dashed lines depicts where the flows splits; (M): A generic branch of those vertices of the form $a^i b a^{p-1}ba^{p'-1}b \cdots$ with $p \le p'$ being primes; (R): A generic branch of those vertices of the form $ b a^{p-1}ba^{p'-1}b \cdots$ with $p \le p'$ being primes.
   }\label{fig:semi}
 \end{figure}

\begin{lemma}\label{cor:Nathanson}
 Recalling $E_n= \{ v \in V: \ \vert v \vert=n\}$, we have
\begin{equation}
\lim_{n \to \infty} \frac{\log \log \# E_n }{\log n}
= \frac{1}{2}.
\end{equation}
\end{lemma}

\begin{proof}
By Theorem \ref{th:Nathanson}, we have
\begin{equation}
\limsup_{n \to \infty} \frac{\log \log \# E_n}{\log n} \le \limsup_{n \to \infty} \frac{\log \log \# B_n}{\log n} = \frac{1}{2}.
\end{equation}

Concerning the lower bound on $\liminf_{n \to \infty} \frac{\log \log \# E_n}{\log n}$,
it is a consequence of the proof in \cite[Theorem 1]{Nathanson1999intermediate}.
For the sake of completeness, we include a proof here.  As observed in \cite[Theorem 1]{Nathanson1999intermediate}, $b^2=b$ and
\begin{equation}
ba^kb=
\begin{pmatrix}
k+1 & 0\\
k+1 & 0
\end{pmatrix},
\end{equation}
and then
for any positive integers $k_1, \cdots, k_r$,
\begin{equation}
b a^{k_1} ba^{k_2}b \cdots ba^{k_r}b=\left( \prod_{i=1}^r(k_i+1) \right) b.
\end{equation}
Taking $p_1,p_2, \cdots, p_r$ to be distinct prime numbers smaller than or equals to $\sqrt{n}$, we have
\begin{equation}\label{distprim}
 b a^{p_1-1} ba^{p_2-1}b \cdots ba^{p_r-1}b= \left( \prod_{i=1}^{r} p_i \right) b.
\end{equation}
Moreover, by Chebyshev's theorem \cite[Theorem 6.3]{Nathanson1996additive}, there exist two universal constants $c,c'>0$ such that
\begin{equation}\label{Chebyshev}
\frac{c' \sqrt{n}}{\log n} \le \# \{p \le \sqrt{n}:\ p \text{ is a prime number } \} \le \frac{c \sqrt{n}}{\log n}.
\end{equation}
By \eqref{Chebyshev},  the length of such word in \eqref{distprim} is bounded above by
$$ 1+ \sum_{i=1}^r p_i \le 1+r \sqrt{n} \le 1+  \frac{ cn }{\log n} \le n$$
for all $n$ sufficiently large.  Then each of such elements belongs to $B_n$. Moreover, any two such elements are distinct since each positive integer is a unique product of primes. Whence, every subset of primes up to $\sqrt{N}$ corresponds to an element of the semigroup $S$. Write
\begin{equation*}
S_m \colonequals \left\{b a^{p_1-1} ba^{p_2-1}b \cdots ba^{p_r-1}b: \ p_1<p_2<\cdots < p_r \text{ are primes}, \forall \; i \le r, \ p_i \le \sqrt{n}; \; \sum_{i=1}^{r}p_i=m  \right\},
\end{equation*}
and we have
\begin{equation}
\sum_{m=0}^{n-1} \# S_m  \ge 2^{c' \sqrt{n}/ \log n}
\end{equation}
where we have used \eqref{Chebyshev}.
Therefore, there exists at least one $m_0 \le n-1$ such that $$\# S_{m_0}  \ge  2^{c' \sqrt{n}/ \log n}/n.$$  Observe that for any two distinct positive integers $k, \ell \ge 2$,  $b a^{k-1}b a^{\ell-1}b =ba^{k \ell-1} b$ and
$$k+ \ell+1= \vert b a^{k-1}b a^{\ell-1}b  \vert \le \vert ba^{k \ell-1} b\vert=k \ell+1.$$
Therefore, every element in $S_m$ is written in the shortest expression.
By $a^{n-m_0-1} S_{m_0} \subset E_n$, we conclude the proof.

\end{proof}

As we know the structure of the lexicographic minimal spanning tree $\cT$, we can calculate its $\IBN$ in the following proposition.

\begin{proposition}\label{prop:mintree}
 For the lexicographic minimal spanning tree $\cT$ of the semigroup $S$ generated by $\{a, b\}$ defined in \eqref{gen:semigr} and setting $b<a$ for the lexicographic order, we have
\begin{equation}
\IBN(\cT)=\frac{1}{2}.
\end{equation}

\end{proposition}

 \begin{proof}
 By \eqref{cutgrow}, we have
\begin{equation}
\IBN(\cT) \le \liminf_{n \to \infty} \frac{\log \log \# E_n}{\log n}=  \frac{1}{2}.
\end{equation}
Now we show $\IBN(\cT)  \ge 1/2$. For any $\gl < \frac{1}{2}$, we designate a flow from $\mathrm{Id}$ to infinity as follows: $\theta(\{\mathrm{Id}, a\})=0,\theta(\{\mathrm{Id}, b\})=c>0$, and the flow along any self-avoiding path from $\mathrm{Id}$ to infinity splits equally when it visits ${p-1}$ the first time with $p$ being prime. We refer to Figure \ref{fig:semi} for a graphic explanation. Now we verify that this flow is admissible, \textit{i.e.} $\theta(e) \le \exp({-\vert e \vert^{\gl}})$. For any vertex at distance $\ell =j+\sum_{i=1}^k p_i$ where $p_1<p_2< \cdots <p_k$ are primes and $0 \le j <\kappa(p_k)$ with
$$\kappa(p_k) \colonequals \inf \{ p>p_k: p \text{ is prime } \} .$$
The flow in the edge $e$ between this vertex and its parent is
\begin{equation}\label{flow:semigr}
 c \cdot 2^{-\# \{ p \le p_k: \ p \text{ is  prime}\}}.
\end{equation}
By Cheyshev's Theorem \cite[Theorem 6.3]{Nathanson1996additive}, for all $p_k$ sufficiently large, we have $q(p_k) \le 2 p_k$  and then
$$\ell \le \kappa (p_k)+ \sum_{i=1}^{p_k}i \le 2p_k +\sum_{i=1}^{p_k}i \le p_k^2.$$
Therefore, by  Cheyshev's Theorem, for all $p_k$ sufficiently large, the r.h.s. in \eqref{flow:semigr} is bounded above by
\begin{equation}
c \cdot 2^{-\frac{p_k}{2\log p_k}} \le c 2^{-\frac{\sqrt{\ell}}{2\log \sqrt{\ell}}} \le e^{-\ell^{\gl}}.
\end{equation}
Moreover, we can choose $c>0$ sufficiently small such that the inequalities above hold for all $p_k$.
Therefore, the designated flow is admissible, and by \cite[Theorem 3.1]{Lyons2016trees} we have $\gl_c \ge \frac{1}{2}$. We conclude the proof.

 \end{proof}

\section{Finding good trees inside permutation wreath products}\label{sec:tree-perwproduct}

In this section we will show how to construct spherically symmetric trees with high IBN inside permutation wreath products of intermediate growth groups. The lower bound we will get on the IBN of the tree will be given in terms of the maximal inverted orbit growth.

\subsection{Permutation wreath products}\label{subse:pwp}
Let $G$ be a  group acting on a set $X$ from the right, and let $A$ be another group. Borrowing terminologies from "lamplighter groups", we will call $G$ the "base" group, and refer to $A$ as the "lamp" group. The permutation wreath product $W=A\wr_X G$ is the semidirect product of $\sum_X A$ with $G$ where $\sum_X A$ consists of those functions $f: X \mapsto A$ with $\supp(f)\colonequals \{ x \in X: \ f(x)\neq 1_{A}\}$ being finite.
A generic element of $W$ is written as $fg$ with  $f \in \sum_X A$ and $g \in G$.
Moreover, for  $f_1,f_2 \in \sum_X A$ and $g_1,g_2 \in G$, the product between $ f_1 g_1$ and $ f_2 g_2$ is defined as
$$ ( f_1 g_1) \cdot (f_2 g_2) \colonequals  (g_2^{-1}\cdot f_1)  f_2  g_1g_2$$
where $(g^{-1} \cdot f)(x) \colonequals f(x g^{-1})$ for all $x \in X$. We distinguish a vertex $x_0 \in X$ and identify $a \in A$ with $f_a \in \sum_X A$ where
\begin{equation}\label{gen2func}
f_a(x) \colonequals
\begin{cases}
a, &\text{ if } x=x_0,\\
1_A, & \text{ if } x \neq x_0.
\end{cases}
\end{equation}
Let $S_A$ and $S_G$ be the generating sets of $A$ and $G$ respectively. We will define a generating set for $W$ by taking all products of the form  $(a_i g_i)_{n \in \bbN}$ with $g_i \in S_G, a_i=1_A$ or $g_i=1_G, a_i \in S_A$. Let $u=a_1 g_1 a_2 g_2 \cdots  a_{\ell} g_{\ell} \equalscolon f_u  g_u $ and $v=u  a_{\ell +1}g_{\ell+1} \equalscolon  f_v g_v$, and then we say
\begin{itemize}
\item[(1)] if $g_{\ell+1} \in S_G$ and $a_{\ell+1}=1_A$, there is an edge of "$G$" type from $u$ to $v$ and $g_v=g_u  g_{\ell+1}, f_v=g_{\ell+1}^{-1} \cdot f_u $;
\item[(2)] if $g_{\ell+1}=1_G$ and $a_{\ell+1} \in S_A$, there is an edge of "$A$" type from $u$ to $v$ and
$g_u=g_v, f_v=f_u$ except at $x_0$ where $f_v(x_0)=f_u(x_0) a_{\ell+1}$.
\end{itemize}
With this rule, the Cayley graph $(W, S_A \cup S_G)$ is determined.
Indeed for $u=a_1 g_1 a_2 g_2 \cdots  a_{\ell} g_{\ell}= f_u  g_u $ , we know that $g_u=g_1 \cdots g_{\ell}$ and $f_u= ((g_1\cdots g_{\ell})^{-1} \cdot f_{a_1}) ((g_2\cdots g_{\ell})^{-1} \cdot f_{a_2}) \cdots  (g_{\ell}^{-1} \cdot f_{a_{\ell}})$. Then we have  $\supp( f_v) \subset \{ x_0g_{\ell}, x_0g_{\ell-1}g_{\ell}, \cdots, x_0 g_1g_2 \cdots g_{\ell}\}$.
For $w=g_1g_2\cdots g_{\ell}$, we define its inverted orbit as
\begin{equation*}
\cO(w) \colonequals \{  x_0g_{\ell}, x_0g_{\ell-1}g_{\ell}, \cdots, x_0 g_1g_2 \cdots g_{\ell}  \}.
\end{equation*}

\subsection{Constructing good trees}

Our main theorem says that if the inverted orbit growth is at least $n^\alpha$ then there is a tree in the Cayley graph with $\IBN \geq \alpha$.

\begin{theorem}\label{th:fightertree}
Assume $\# S_A \ge 2$ and that there exists $C_2>0$ such that for all $n \in \bbN$ there is a word $w_n$ of length $n$ satisfying
\begin{equation}\label{assum2}
\#\cO(w_n)  \ge C_2 n^{\ga}.
\end{equation}
Then there is a  spherically symmetric tree $\cT$ in the Cayley Graph
$(W,S_A \cup S_G)$ satisfying $$\IBN(\cT) \ge \ga.$$
\end{theorem}

\subsection*{Proof of Theorem \ref{th:fightertree}}
We divide the proof into several steps.\\
\textbf{Step 1}: We first find an infinite word $w=g_1g_2\cdots$ satisfying
\begin{equation}\label{lwbd:inforbit}
 \text{for all } n \in \bbN, \quad   \#\cO(g_1\cdots g_n) \geq Cn^{\alpha} .
\end{equation}

By \eqref{assum2},  for all $n \in \bbN$ there exists a word $\xi_n=g_1^{(n)}g_2^{(n)}\cdots g_{2^n}^{(n)} \in G$ with $g_i^{(n)} \in S_G $ such that $\# \cO (\xi_n) \ge C_2 2^{\ga n}.$
We construct an infinite word $w=\xi_1 \xi_2 \xi_3 \cdots \xi_n \cdots$
 by concatenating $(\xi_n)_{n \in \bbN}$.  For $0 \le i <2^{n+1}$,  observe
 \begin{equation}\label{inclusionIO}
 \begin{gathered}
 \cO(\xi_1 \xi_2 \cdots \xi_n g^{(n+1)}_1\cdots g^{(n+1)}_i) \supset \cO (\xi_n)g^{(n+1)}_1\cdots g^{(n+1)}_i,\\
\vert \xi_1 \xi_2 \cdots \xi_n g^{(n+1)}_1\cdots g^{(n+1)}_i \vert \le 4 \vert \xi_n \vert,
\end{gathered}
 \end{equation}
 where $|\xi|$ denotes the number of generators in the word $\xi$. Therefore $w=\xi_1 \xi_2 \cdots$ satisfies our requirements.

\textbf{Step 2}:
We now want to use the infinite word $w$ and the wreath structure of the permutation wreath product to construct our tree. The basic idea is to use the lamps to branch the path. We want to split 
whenever we encounter a "new" branching position 
that is when the inverted orbit increases by $1$. However, if the orbit of a segment inside the infinite word is a loop in the base group, then this may cause a loop in our constructed graph, so that it may not be a tree. An important observation is that the problematic loops are only those in which the size of the inverted orbit does not increase (otherwise, the branching breaks the loop in $W$.)
We therefore describe a way to "delete" those loops where the inverted orbit does not increase. We make this more precise below.

We start with the following observation
\begin{equation}\label{build:IO}
 \cO(g_1 \cdots g_n g_{n+1})= \left\{ x_0 g_{n+1}\right\} \cup \cO(g_1 \cdots g_n) g_{n+1},
\end{equation}
and in particular $\# \cO(g_1 \cdots g_n g_{n+1}) \ge \# \cO(g_1 \cdots g_n)$.
Moreover, for a loop $g_{i+1}g_{i+2} \cdots g_{i+j}=1_{G}$, we say that this loop does not contribute to the inverted orbit if
\begin{equation}\label{unuse:loop}
\text{for all } \, k \in \lint 1, j\rint, \quad
 x_0g_{i+k} \in \cO(g_1\cdots g_i)g_{i+1} \cdots g_{i+k}.
\end{equation}
In this case, we have
\begin{equation}\label{orbit:erase}
\cO(g_1\cdots g_i g_{i+1} \cdots g_{i+j} g_{i+j+1})= \cO(g_1\cdots g_i g_{i+j+1}).
\end{equation}

This implies that the erasure of such a loop does not affect the inverted orbit of any prefix containing this loop, and therefore  we can delete it from the word $w$ without affecting the inverted orbit growth.
Notice on the other hand, that if a loop  does not satisfy \eqref{unuse:loop}, we will not 
delete it, as deleting such a loop will decrease the inverted orbit. Such loops however will not cause a problem as they add at least one new point in the inverted orbit which we use to branch and thus break the loop in the permutation wreath product.

\textbf{Step 3}: We next give an algorithm for erasing only those loops in $w$ satisfying \eqref{unuse:loop}. We use the convention $\inf \emptyset= \infty.$

\begin{lemma} \label{lema:goodloop}
There exists an infinite word $ q= q_1 q_2 \cdots$ satisfying: for all $n \in \bbN$
\begin{equation}\label{qorbsize2}
   \# \cO(q_1 \cdots  q_n) \ge C_2 \left(\frac{n}{4}\right)^{\ga},  
\end{equation}
and if $ q_{i+1} \cdots  q_{i+j}=1_{G}$ is a loop, it does not satisfies \eqref{unuse:loop}. That is to say, any loop in $ q$ increase the size of the inverted orbit at least by one.
\end{lemma}

\begin{proof}
We take the word $w=g_1 g_2 \cdots g_n \cdots$ stated in
Step 1
and delete all those loops satisfying \eqref{unuse:loop} in order of appearance, to obtain the word $ q$.
We locate these kind of loops by specifying the starting and ending indices. The starting and ending points of the first loop are
\begin{equation*}
\begin{gathered}
L_1 \colonequals \inf \{ n \ge 1: \exists k \in \bbN, g_n g_{n+1} \cdots g_{n+k}=1_G \text{ and } \# \cO(g_1 \cdots g_{n+k}) = \# \cO(g_1 \cdots g_{n-1})  \},
\\
U_1 \colonequals \sup \{ \ell> L_1: g_{L_1} g_{L_1+1}\cdots g_\ell=1_G \text{ and } \# \cO(g_1 \cdots  g_{\ell})  = \# \cO(g_1 \cdots g_{L_1-1})  \},
\end{gathered}
\end{equation*}
where $\cO(g_1 \cdots g_{n-1}) \colonequals \{ x_0\}$ for $n=1$. If $L_1= \infty$, there is no loop satisfying \eqref{unuse:loop} and then we do not need to erase any loop in $w$. If $L_1< \infty$, by \eqref{orbit:erase} and \eqref{lwbd:inforbit} we have $U_1< \infty.$
Supposing $(L_i,U_i)_{1\le i<m}$ is defined with $U_{m-1}< \infty$, we define
\begin{equation*}
\begin{gathered}
L_m \colonequals \inf \{ n > U_{m-1}: \exists k \in \bbN, g_n g_{n+1} \cdots g_{n+k}=1_G \text{ and } \# \cO(g_1 \cdots g_{n+k})  = \# \cO(g_1 \cdots g_{n-1})  \},
\\
U_m \colonequals \sup \{ \ell> L_m: g_{L_m} g_{L_m+1}\cdots g_\ell=1_G \text{ and } \# \cO(g_1 \cdots  g_{\ell})  = \# \cO(g_1 \cdots g_{L_m-1})  \}.
\end{gathered}
\end{equation*}
If $L_m= \infty$,  then $(L_i,U_i)_{1\le i<m}$ is the final sequence. Else if $L_m< \infty$, then  $U_m< \infty$
 by \eqref{orbit:erase} and \eqref{lwbd:inforbit}. We can iterate this algorithm until some $L_i=\infty$ or we obtain an infinite sequence $(L_i,U_i)_{i \in \bbN}$ with $\lim_{i \to \infty} L_i= \infty.$
This algorithm is well-defined:  for any prefix $g_1g_2 \cdots g_n$, by \eqref{lwbd:inforbit} we delete at most $n-Cn^{\ga}$
letter  since we only erase those loops satisfying \eqref{unuse:loop}.
Therefore, for all $n$ sufficiently large, and for a prefix $g_1 \cdots g_n$ of $w$,  to decide which letters in $g_1 \cdots g_n$ will be deleted,  we just need to run the algorithm with the  prefix $g_1 \cdots g_n g_{n+1} \cdots g_{2n}$.
 After the loop erasure procedure in $q$, we obtain an infinite word $ q$ without loops satisfying \eqref{unuse:loop}:
$$  q = g_1 g_2 \cdots g_{L_1-1} g_{U_1+1}g_{U_1+2} \cdots g_{L_2-1} g_{U_2+1}g_{U_2+1} \cdots g_{L_i-1} g_{U_i+1} g_{U_i+2}\cdots g_{L_{i+1}-1} \cdots .$$
In words, we delete the loops $(g_{L_i} g_{L_i+1} \cdots g_{U_i})_{i \in \bbN}$ in $q$ to obtain $ q.$
We relabel the indices to write  $ q=  q_1 q_2 \cdots q_n \cdots $. The lower bound in
\eqref{qorbsize2}  follows from \eqref{lwbd:inforbit} and the fact that
the algorithm does not delete any invert orbit point in any prefix of $q$.
\end{proof}

\textbf{Step 4}: With $ q$ at hand, we are ready to construct a spherically symmetric tree $\cT$ in the Cayley graph $(W, S_A \cup S_B)$ by sequentially choosing the vertices for $\cT$ layer by layer, and linking vertices to the previous layer by an edge of type "$A$" or "$G$" according to the rules in Subsection \ref{subse:pwp}.

We fix two generators $a_1,a_2 \in S_A$ and let $*$ represent a generic element in $\{ a_1,a_2\}$.
The algorithm for picking the vertices for $\cT$  is as follows:
as we walk along the word $ q$ (in the Cayley $(G,S_G)$), we place $*$ if and only if the size of the inverted orbit increases by one. 
The root of the tree is $1_W$, and $a_1, a_2 \in V(\cT)$. We mark those positions where the size of the inverted orbit increases by one when we go along the word $ q$ as follows:
$M_1 \colonequals 1$ and for all $n \ge 2$,
\begin{equation}
M_n \colonequals \inf \left\{\ell: \# \cO(q_1 \cdots q_{\ell}) \ge  n \right \}.
\end{equation}
For any $n$, we define
\begin{equation}
\zeta(n) \colonequals \sup \left\{ \ell: M_{\ell} \le n \right\},
\end{equation}
and then a generic vertex in $V(\cT)$ is of the form:
$$*q_1q_2\cdots q_{M_2-1}*q_{M_2}\cdots q_{M_3-1}*q_{M_3}q_{M_3+1} \cdots q_{M_4-1}*q_{M_4}\cdots q_{M_{\zeta(n)}-1} *q_{M_{\zeta(n)}}q_{M_{\zeta(n)}+1}\cdots q_n .$$
In other words, for any prefix of $ q$, we add a generic lamp generator $* \in \{ a_1, a_2\}$ exactly in front of those $q_i$s where the size of the inverted orbit increases by one. Note that after adding $*$ before $(q_{M_i})_{i \in \bbN}$ in the loops of $ q$ (in the Cayley graph $(G,S_G)$), there is no loop in $\cT$ any more.

Now we can calculate $\IBN(\cT)$. Since by construction $\cT$ is a spherically symmetric tree, we have $\IBN(\cT)=\Igr(\cT)$.
To calculate $\Igr(\cT)$, we estimate $\# E_N $ where $E_n=\{x \in V(\cT): \ \vert x \vert=n \}$ with the graph distance $\vert \cdot \vert$ in $\cT$. At distance $n$ in $\cT$, we count the number of splittings
\begin{equation}
\gL(n) \colonequals     \sup \left\{ \ell: \vert \cO(q_1 \cdots q_{\ell}) \vert+ \ell \le n \right\}
\end{equation}
Since $\vert \cO(q_1 \cdots q_{\ell}) \vert \le \ell$, we have $\gL(n) \ge (n-2)/2$.
As $\# E_n = 2^{\vert \cO(q_1 \cdots q_{\Lambda(n)})  \vert}$, by \eqref{qorbsize2}   we have
\begin{equation}
\liminf_{n \to \infty} \frac{\log \log \# E_n }{\log n}\geq \ga.
\end{equation}
\qed

\subsection{A tree in the wreath permutation group based on the first Grigorchuk group.} \label{se:firefighterexample}
The first Grigorchuk group $G_{012}$ is a subgroup of the automorphism group of the binary tree. It
 is  generated by $\{ a,b,c,d\}$ which are defined by the wreath recursion  as follows: let $ \gep$ be the transposition of $(0, 1)$, and the wreath recursion $\psi$ is defined as
\begin{equation}\label{gri:recur}
\psi: a \mapsto \gep \llangle  1,1 \rrangle , \quad b \mapsto \llangle a, c \rrangle, \quad c \mapsto \llangle a, d \rrangle, \quad d \mapsto \llangle 1, b\rrangle.
\end{equation}
The generators $\{ a,b,c,d\}$ satisfy
\begin{equation}\label{pre:rule}
a^2=b^2=c^2=d^2=1, \quad bc=cb=d, \quad bd=db=c, \quad cd=dc=b.
\end{equation}
We refer to \cite{Grigorchuk2008introduction} and \cite[Chapter 10]{Mann2012grow} for introductions to the first Grigorchuk group.
Bartholdi and Erschler prove the following theorem concerning the inverted orbit for $G_{012}$ acting on the rightmost branch of the binary tree, \textit{i.e.} $x_0=1^{\infty}$ where the binary tree is labeled by $\{0,1 \}^{\infty}$ with two vertices linked by an edge if one is obtained by appending a "$0$" or "$1$" to the other.

\begin{theorem}
[  {\cite[Proposition 4.7]{Bartholdi2012permutational}}] \label{th:BarErs}
Let $\eta$ be the real root of the polynomial $X^3+X^2+X-2$ and set $\ga=\log 2/ \log(2/\eta) \approx 0.7674$. Taking $G=G_{012}$, $A$ some finite group, $W=A\wr_X G_{012}$ and $x_0=1^{\infty}$, then there exists some $C>0$ such that for any $n\geq 1$ there exists a word $w_n\in G$ of length $n$ with $\#\cO(w_n)  \ge C n^{\ga}$.
\end{theorem}
Note that by  \cite[Proposition 4.4]{Bartholdi2012permutational}, for any word $w\in G$ we have $\#\cO(w_n)  \leq C' |w|^{\ga}$.
By \cite[Theorem 5.2]{Bartholdi2012permutational}, for any finite group $A$ with $\#A \ge 3$, the growth rate of the permutation wreath product $W=A\wr_X G_{012}$ satisfies $\# B_n \approx e^{n^\ga}$ (where  $B_n$ denotes the ball of radius $n$ in $W$).
Combining this with Theorem \ref{th:fightertree} we obtain the following corollary.

\begin{cor}\label{cor:grigtree}
For any finite group $A$ with $\#A \ge 3$, there exists a spherically symmetric tree $\cT$ inside the Cayley graph of $W=A\wr_X G_{012}$ that satisfies $\IBN(\cT)=\Igr(W) = \ga$.
\end{cor}

We remark that the above construction and corollary can be extended to more general Grigorchuk groups $G_{\omega}$, such as the groups used in \cite{Bartholdi2014growth}.

\section{The firefighting problem on intermediate growth trees and groups}\label{sec:fireTreeGroup}
\subsection{The firefighting problem}
Let $\cG=(V,E)$ be a locally finite, connected, infinite graph, with $E$ denoting its edges and $V$ denoting its vertices.  Given $g=(g_n)_n$ a sequence of nonnegative integers, we play a firefighting game on $\cG$ as follows: ($W_0 \colonequals \emptyset$)
\begin{itemize}
\item[1.] the initial fire is set on a finite subset $Z_0 \subset V$;
\item[2.] for $n \ge 1$, at round $n \in \bbN$, we can pick a subset $S_n \subset V \setminus (Z_{n-1} \cup W_{n-1}) $ with $\# S_n   \le g_n$ to protect all the vertices in $S_n$. Once a vertex is protected (or on fire), it keeps the state forever.
\item[3.] for $n \ge 1$, set
\begin{equation}
\begin{gathered}
W_n=W_{n-1} \cup S_n= \cup_{i=1}^n S_n,\\
Z_n=Z_{n-1} \cup \left\{ v \in V\setminus W_n: \ \exists u \in Z_{n-1}, u \sim v \right\}.
\end{gathered}
\end{equation}
That is to say, $W_n$ are the protected vertices up to round $n$, and in every round the fire spreads to its nearest neighbors except those protected vertices.
\end{itemize}
Given $Z_0$,  let  $\Upsilon: \bbN \mapsto V$ be the map defined by $\Upsilon(n)=S_n$ described above, and we say that $\Upsilon$ is legal. Furthermore, if  the fire stops growing after finitely many step, \textit{i.e.} $\# \cup_{n=0}^{\infty} Z_n  < \infty$,  we say that $\Upsilon$ is containable for $Z_0$ or the fire $Z_0$ is containable w.r.t. $\Upsilon$.
Otherwise, we say that the fire percolates.
Furthermore, the graph $\cG$ is said to be $g$-containable if for any finite subset $Z_0$ there is a corresponding containable map $\Upsilon$ satisfying $ \# \Upsilon(n) \le g(n)$ for all $n \in \bbN$. Understanding the relation between $\cG$ and its containment functions is considered an asymptotic version of the firefighting game introduced by Hartnell in \cite{hartnell1995firefighter}. It is interesting to note that containment is a quasi-isometric invariant (see \cite{Dyer2017coarse}).

The firefighter problem is of particular interest in the case of Cayley graphs, where one can hope to find connections between the geometry of the group $\cG$ and the containment threshold. Precise asymptotic bounds are known for exponential growth groups (\cite{Lehner2019firefighting}) and for polynomial growth groups (\cite{AmBaKofire}), but no such tight bounds were known for any intermediate growth groups.
Our main aim in this section is to find such bounds for a family of intermediate growth groups. The groups we will use are permutation wreath products. More precisely our bounds will be tight for a certain family of permutation wreath products over Grigorchuk's group, studied by Bartholdi and Erschler \cite{Bartholdi2012permutational}.

\subsection{Firefighting on intermediate growth graphs and groups}
A graph $\cG$ satisfies intermediate containment of rate $\gga \in (0,1)$ if there is some $C>0$ and $g=(g_n)_{n \in \bbN}$ with $g_n \le C \exp(n^{\gga})$ for all $n \in \bbN$ such that $\cG$ is $g$-containable.
Due to the monotonicity (cf. Lemma \ref{lema:initialball}), there is a critical parameter $\gl_c$ such that if $\gga> \gl_c$ then $\cG$ satisfies intermediate containment of rate $\gga$  and if $\gga< \gl_c$ then $\cG$  does not satisfy intermediate containment of rate $\gga$. For $\gl_c=0$ or $\gl_c \ge 1$, the graph $\cG$ is not of intermediate containment type in our definition.
Our first theorem on firefighting says that for intermediate growth trees, the intermediate branching number is the critical parameter for containment. This can be seen as an analogous result to Lehner's theorem for exponential growth trees \cite[Theorem 3.1]{Lehner2019firefighting}, and we delay the proof to the Appendix \ref{sec:fireIBN}.

\begin{theorem}\label{th:fireintermediate}
Let $\cT$ be a locally finite, infinite tree satisfying $\IBN(\cT)= \ga \in (0,1)$,
and then we have $\IBN(\cT)= \gl_c(\cT)$.
\end{theorem}

It is easy and well-known that for any group, the growth rate of the spheres is an upper bound for the firefighting problem. Using the above theorem and monotonicity, it follows that if one can find a tree $\cT$ inside the Cayley graph $\cG=(G,S_G)$ of $G$ with $\IBN(\cT)=\Igr(\cG)$ then the critical threshold for figherfighting on $\cG$ will be $\IBN(\cT)$. A similar strategy was done for the exponential case in \cite{Lehner2019firefighting}, using lexicographical minimal spanning trees. The lexicographical spanning tree is sub-periodic, and by a theorem of Furstenberg (cf. \cite{Furstenberg67} \cite[Theorem 3.8]{Lyons2016trees}) sub-periodic trees have a (exponential) branching number equal to their growth rate. However, for the IBN, sub-periodicity does not imply that the intermediate branching number is equal to the growth (see Theorem \ref{th:tang}). Therefore, instead, we use the trees constructed in section \ref{th:fightertree}.

For the case of permutation wreath products over Grigorchuk groups, the upper and lower bounds coincide, and we can get the exact threshold for the firefighting problem.
More precisely, by  \cite[Theorem 5.2]{Bartholdi2012permutational}, when $A$ is a finite group with $\# A \ge 3$, the growth rate of the permutation wreath product $W=A\wr_X G_{012}$ satisfies $\# B_n \approx e^{n^\ga}$, and by Corollary \ref{cor:grigtree}, there is a subtree $\cT$ of the Cayley graph of $W$ with $\IBN(\cT)=\ga$. Therefore by the above paragraph we get

\begin{cor} Let $A$ be a finite group with $\#A \ge 3$. 
The critical threshold for firefighter problem on the permutation wreath product $W=A\wr_X G_{012}$ is $\ga$.
\end{cor}

 As in the previous section, we note that the above reasoning and corollary can be extended to more general Grigorchuk groups $G_{\omega}$ as in \cite{Bartholdi2014growth}. 

\section{Proof of Proposition \ref{th:intersymsub}}
\label{sec:noperiodtree}
\begin{proof}
Suppose there exits  such a tree, and it is $M$-subperiodic.  Let $(a_n)_{n \in \bbN_0}$ denote a sequence of positive integers such that a vertex at generation $n$ has $a_n$ children.
By $M$-sub-periodicity, for all $j >M$, there exists $i=i(j) \le M$ such that
\begin{equation}\label{subperi:relation}
\forall n \in \bbN_0, \quad a_{j+n} \le a_{i+n}.
\end{equation}
In particular, $$ \sup_{n \in \bbN_0}a_n= \max_{0 \le n \le M}a_n \colonequals U \ge 2. $$
Therefore,
$$  2^{\sum_{n=0}^{n-1} \ind_{\{ a_n \ge 2\}}} \le  \# E_n  \le U^{\sum_{n=0}^{n-1} \ind_{\{ a_n \ge 2\}}}. $$
Therefore it is enough to analyze a simpler sub-periodic spherical symmetric tree $\cT'$ associated with the sequence $(b_n)_{n \in \bbN_0}$ where
$$ \forall n \in \bbN_0, \quad b_n \colonequals 2 \ind_{\{ a_n \ge 2\}}+\ind_{\{ a_n=1\}}$$ is $M$-subperiodic and $\lim_{n \to \infty} E_n(\cT')=\infty$.
From now on, for simplicity, we assume $a_n \in \{ 1,2\}$ for all $n \in \bbN_0$, and let $c_n=a_n-1$.

\medskip
For $\gep>0$ sufficiently small,  by \eqref{assum:subexp}, there exists $N \in \bbN$ sufficiently large such that
$$  \sum_{k=0}^{N-1} c_k \le \gep N . $$
Taking $\ell=M+1 \ge 1$, we cut the interval $\lint 0, \ell \lfloor N/\ell \rfloor-1 \rint$ into intervals of size $\ell$, and count the number of those intervals without $c_n$ taking value $1$, \textit{i.e.}
\begin{equation}\label{ct:inervals}
\# \cA \colonequals \# \left\{j \in \lint 0, \ell \lfloor N/\ell \rfloor-1 \rint: \  \forall n \in \lint j \ell, (j+1)\ell-1 \rint, c_n=0 \right\}  \ge \left\lfloor \frac{N}{\ell}   \right\rfloor-\gep N \ge \frac{N}{2\ell}.
\end{equation}
Setting $\kappa \colonequals \min\{j: \ j \in \cA\}$,  we claim
\begin{equation}\label{claim:no1}
\left\{ n \in \lint \kappa \ell, \ell \lfloor N/\ell \rfloor-1 \rint:\ c_n=1 \right\}= \emptyset.
\end{equation}
Otherwise, let $n_0= \min \left\{ n \in \lint \kappa \ell, \ell \lfloor N/\ell \rfloor-1 \rint:\ c_n=1 \right\} $, and note that $n_0 \ge (\kappa+1)\ell$ by the definition of $\kappa$.

\begin{figure}[h]
 \centering
   \begin{tikzpicture}[scale=.4,font=\tiny]
    \draw (25,-1) -- (52,-1);
    \draw[->] (25,-1) -- (52,-1) node[anchor=north west]{x};
     \foreach \x in {25,...,51} {\draw (\x,-1.3) -- (\x,-1);}

     \draw[red, thick] (25,-1)--(28,-1);
     \draw[red, thick] (29,-1)--(32,-1);
     \draw[red, thick] (33,-1)--(36,-1);
     \draw[blue, thick] (37,-1)--(40,-1);
     \draw[red, thick] (41,-1)--(44,-1);

     \foreach \x in {25,26,28,31,35,36,43} {\draw[fill] (\x,-1) circle [radius=0.1];}
    \foreach \x in {41,42} {\draw[fill, blue] (\x,-1) circle [radius=0.1];}
     \node[below] at (43,-1.3) {$n_0$};

    \draw[blue, thick] (45,-1)--(48,-1);

     \node[left] at (25,-1) {$0$};
     \node[below] at (51,-1.3) {$N$};

   \end{tikzpicture}
   \caption{\label{fig:spheresym}  A graphical explanation with $\ell=\kappa=4$: black dots are the positions where $c_n=1$, the other positions (including blue dots) are where $c_n=0$.
       Blue intervals are those in $\cA$.  By subperiodicity, we can move the black dot at $n_0$ to the left at least $1$ and at most $\ell$, and thus it should be blue. }
 \end{figure}
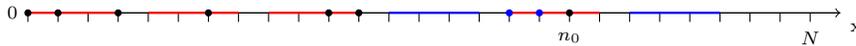
\noindent
Taking $j=M+1$ and applying the analogous relation \eqref{subperi:relation} for $(c_n)_{n \in \bbN_0}$, we obtain
\begin{equation}\label{moveblog}
 c_{n_0} \le c_{n_0-(j-i)} =0
\end{equation}
where $i=i(M+1)$ is a number depending on $M+1$ as claimed before \eqref{subperi:relation}, and the equality is a consequence of $0<j-i \le M+1$,
$c\restriction_{\lint \kappa \ell, n_0-1 \rint}=0$ and
$$\kappa \ell \le (\kappa+1)\ell-(M+1) \le n_0-(j-i) \le n_0-1.$$
 This is a contradiction to the assumption $c_{n_0}=1$. Therefore, we prove  \eqref{claim:no1}.

By induction and \eqref{moveblog}, we obtain $c_n=0$ for all $n \ge \kappa \ell.$
This contradicts the assumption $\lim_{n \to \infty} \# E_n = \infty$.
Therefore, we conclude the proof.
\end{proof}

\

 \subsection*{\bf{ Acknowledgment}}
The authors thank Mikhail Belolipetsky, Russell Lyons and Olivier Thom for helpful discussions, and we acknowledge Daniel Kious and Russell Lyons for comments on references.  G.A. is supported by the Israel science foundation grant 957/20, and S.Y. is supported by the Israel science foundation grants  1327/19 and 957/20.
This work has been performed partially when S.Y. was a postdoc fellow in UFBA and a visiting fellow in IMPA.

\begin{appendix}

\section{ Proof of Theorem \ref{th:IBNrw}} \label{sec:rwphase}

We first deal with the statement about recurrence. Since $\gl> \IBN(\cT)$, then
\begin{equation}
\inf_{\pi \in \Pi} \sum_{e \in \pi} \exp \left( - \vert e \vert^{\gl}\right)=0.
\end{equation}
Therefore, there exists a sequence of pairwise disjoint cutsets $(\pi_n)_{n \ge 1} \subset \Pi$ with $\pi_n \cap \pi_m=\emptyset$ for $n \neq m$, satisfying
$$\forall \ n \in \bbN, \quad \sum_{e \in  \pi_n}\exp \left(- \vert  e \vert^{\gl} \right) \le \frac{1}{n}.$$
We apply the Nash-Williams Criterion \cite[Page 37]{Lyons2016trees} to give a lower bound on the effective resistance, defined in \cite[Eq. (2.4)]{Lyons2016trees}, between $\varrho$ and infinity:
\begin{equation}
\cR( \varrho \leftrightarrow \infty) \ge \sum_{n=1}^{\infty} \left(  \sum_{e \in \pi_n} \exp \left(-\vert e \vert \right)^{\gl} \right)^{-1}= \infty.
\end{equation}
Thus, $\textbf{X}$ is recurrent.\\

Now we move to treat the statement about transience by applying \cite[Proposition 3.4]{Lyons2016trees}. 
We take $\gl' \in (\gl, \IBN(\cT))$ and define
$w_n \colonequals \exp \left( -n^{\gl'}+n^{\gl} \right)$ for $n \in \bbN$. We have
$$ \sum_{n=1}^{\infty} w_n < \infty, $$
and $w_{\vert e \vert} c(e)= \exp(- \vert e \vert^{\gl'}) $. By \cite[Theorem 3.1]{Lyons2016trees}, we have
\begin{equation}
\begin{aligned}
&\max \left\{ \mathrm{Strength}(\theta): \ \theta \text{ flows from $\varrho$ to $\infty$ satisfying } \forall e \in E, \ 0 \le \theta(e) \le c(e) w_{\vert e \vert} \right\}\\
=&\inf_{\pi \in \Pi} \sum_{e \in \pi} c(e) w_{\vert e \vert}>0
\end{aligned}
\end{equation}
where the last inequality uses $\gl'< \IBN(\cT)$.
By \cite[Proposition 3.4]{Lyons2016trees}, there exists a flow $\theta$ satisfying $0 \le \theta(e) \le w_{\vert e \vert} c(e)$ for all $e \in E$ whose  energy is finite. Therefore, by \cite[Theorem 2.11]{Lyons2016trees}, $\textbf{X}$ is transient.
\qed

\section{Proof of Theorem \ref{th:percolate}}\label{sec:percolate}

We first show that $\Theta(\cT) \le \IBN(\cT)$ by the first moment method as in \cite[Proposition 5.8]{Lyons2016trees}.
For $x \in V$, let $e(x)$ denote the edge connecting $x$ with its parent $\accentset{\leftharpoonup}{x}$.
For any cutset $\pi$ separating $\varrho$ from $\infty$, we have
\begin{equation*}
 \{ \varrho \leftrightarrow \infty \} \subset \bigcup_{x:\ e(x)\in \pi}\{ \varrho \leftrightarrow x\},
\end{equation*}
and then
\begin{equation}
\bbP \left[ \varrho \leftrightarrow \infty \right] \le \inf_{\pi \in \Pi} \sum_{x: \ e(x)\in \pi} \bbP \left[ \varrho \leftrightarrow x \right]
\end{equation}
where $\Pi$ is the set of cutsets. Moreover, $\bbP \left[ \varrho \leftrightarrow x \right]=\prod_{y \le x} \exp(-\vert y\vert^{-1+\gl})$, and there exist two constant $c(\gl), C(\gl)>0$ such that for all $n \in \bbN$
$$ c(\gl) \le \frac{1}{n^{\gl}} \sum_{i=1}^n i^{-1+ \gl} \le C(\gl). $$
Then for $\gl> \IBN(\cT)$, we have
\begin{equation}
\bbP \left[ \varrho \leftrightarrow \infty \right] \le \inf_{\pi \in \Pi} \sum_{x: \ e(x)\in \pi} \exp \left(  -c \vert x\vert^{\gl}\right)=0.
\end{equation}
Therefore, $\Theta(\cT) \le \IBN(\cT)$.

\medskip
Now we turn to show $\theta(\cT) \ge \IBN(\cT)$ by applying \cite[Theorem 5.14]{Lyons2016trees}, \textit{i.e.}
\begin{equation}
\bbP \left[ \varrho \leftrightarrow \infty \right] \ge \frac{\cC(\varrho \leftrightarrow \infty)}{1+ \cC(\varrho \leftrightarrow \infty) }
\end{equation}
where, as in \cite[Equation (5.13)]{Lyons2016trees}, the conductance $c(e(x))$ of the edge $e(x)$ is defined by
\begin{equation}\label{perc:conductance}
c(e(x))\colonequals \frac{\bbP \left[ \varrho \leftrightarrow x \right]}{1-\bbP \left[ e(x) \text{ is open}\right]}=\frac{1}{1-\exp \left( -\vert x\vert^{-1+\gl} \right)}\prod_{\varrho< y \le x} \bbP \left[ e(y) \text{ is open }\right].
\end{equation}
Whence, it suffices to show that for $\gl \in (0, \IBN(\cT))$, we have $\cC(\varrho \leftrightarrow \infty) >0.$
By \eqref{perc:conductance}, for $\gl' \in (\gl, \IBN(\cT))$, there exists a constant $c>0$ such that for all $x \neq \varrho$
$$c(e(x)) \ge \exp \left( - C(\gl) \vert x \vert^{\gl} \right) \ge c \exp(-\vert  x\vert^{\gl'}).$$
By Theorem \ref{th:IBNrw} and Rayleigh's Monotonicity Principle, we obtain $\cC(\varrho \leftrightarrow \infty)>0$ to conclude the proof.

\qed

\section{Proof of Theorem \ref{th:RWRCIBN}}
\label{sec:rwrandcond}
We define $\psi_{RC}: E \mapsto [0,1]$ by $\psi_{RC}(e)=1$ for $\vert e \vert =1$,  and for $\vert e \vert >1$
\begin{equation}
\psi_{RC}(e)\colonequals
\frac{\sum_{g<e} C_g^{-1}}{\sum_{g \le e} C_g^{-1}}.
\end{equation}
which is the probability that the RWRC, restricted to $[\varrho, e^+]$ and starting from $e^-$, hits $e^+$ before hitting $\varrho$.
Let
\begin{equation}
\Psi_{RC}(e) \colonequals \prod_{g \le e} \psi_{RC}(e)=\frac{1}{\sum_{g \le e}C_g^{-1}}
\end{equation}
which is the probability of RWRC, restricted to $[\varrho, e^+]$ and starting at $\varrho$, hits $e^+$ before returning to $\varrho$.
For any function $\psi: E \mapsto [0,1]$,   setting $\Psi(e)\colonequals \prod_{g \le e}\psi(g)$ and
\begin{equation}
RT(\cT, \psi) \colonequals \sup \left\{\gamma > 0:  \inf_{\pi \in \Pi}\sum_{e \in \pi} \Psi(e)^{\gamma}>0 \right\},
\end{equation}
we state   \cite[Theorem 2.5]{Collevecchio2020branching} in our setting.

\begin{proposition}[Theorem 2.5 in \cite{Collevecchio2020branching}] \label{prop:CKSthm}
For any locally finite, infinite tree $\cT$, we have
\begin{itemize}
\item if $RT(\cT, \psi_{RC})<1$ $\bbP$-a.s., then RWRC associated with $\psi_{RC}$ is recurrent;
\item if $RT(\cT, \psi_{RC})>1$ with positive $\bbP$-probability, then RWRC associated with $\psi_{RC}$ is transient.
\end{itemize}
\end{proposition}

\subsection{Recurrence}

\begin{proposition}\label{prop:recurRWRC}
If $\gl> \ga=\IBN(\cT)$, then RWRC associated with $\psi_{RC}$ is recurrent almost surely.
\end{proposition}

\begin{proof}

For fixed $\gep>0$ sufficiently small, we want to apply Proposition \ref{prop:CKSthm} for which we need to show
\begin{equation}
\inf_{\pi \in \Pi} \sum_{e \in \pi}  \Psi_{RC}(e)^{1- \gep}=0.
\end{equation}
We first show that typically $\Psi_{RC}(e)$ is small:
\begin{multline}
\bbP \left[  \sum_{g \le e} C_g^{-1} \le \exp\left(  \vert  e \vert^{\gl} \right) \right] \le
 \bbP \left[  \bigcap_{g \le e} \left\{  C_g^{-1} \le \exp\left(  \vert  e \vert^{\gl} \right)  \right\} \right]\\
= \prod_{g \le e} \bbP \left[   \left\{  C_g^{-1} \le \exp\left(  \vert  e \vert^{\gl} \right)  \right\} \right] = \left(1 -  \bbP\left[  C_e < \exp \left(- \vert e \vert^{\gl} \right)  \right]\right)^{\vert e \vert}
\le \exp \left( - L(\vert e \vert) \vert e \vert^{\gl} \right)
\end{multline}
where we have used \eqref{asm:conduct} and the inequality $1-x \le \exp(-x)$ for all $x \in [0,1]$ in the last inequality.

Taking $\gep>0$ such that  $\ga + 2 \gep< \gl$, by definition of $\IBN$ in \eqref{def:IBN} there exists a sequence of cutsets $(\pi_n)_n \subset \Pi$ satisfying
\begin{equation}\label{choose:cutsets}
\forall n \in \bbN, \quad \quad \sum_{e \in \pi_n} \exp \left( - \vert e \vert^{\ga+ \gep} \right)\le \exp(-n).
\end{equation}
 Concerning $(\pi_n)_n \subset \Pi$, we have
\begin{multline}
\bbP \left[  \bigcup_{e \in \pi_n} \left\{\sum_{g \le e} C_g^{-1} \le \exp \left( \vert e \vert^{\gl} \right) \right\} \right]
\le \sum_{e \in \pi_n}  \bbP \left[   \sum_{g \le e} C_g^{-1} \le \exp \left( \vert e \vert^{\gl} \right)  \right]
\le \sum_{e \in \pi_n} \exp \left( - L(\vert e \vert) \vert e \vert^{\gl} \right),
\end{multline}
and then
\begin{multline}
\sum_{n=1}^{\infty} \bbP \left[  \bigcup_{e \in \pi_n} \left\{\sum_{g \le e} C_g^{-1} \le \exp \left( \vert e \vert^{\gl} \right) \right\} \right] \le \sum_{n=1}^{\infty} \sum_{e \in \pi_n} \exp \left( - L(\vert e \vert) \vert e \vert^{\gl} \right)
\le  \sum_{n=1}^{\infty} \sum_{e \in \pi_n} C \exp \left( - \vert e \vert^{\ga+ \gep} \right)< \infty,
\end{multline}
where we have used \eqref{choose:cutsets} and $\ga +\gep < \gl$.
By Borel-Cantelli Lemma, for almost surely all fixed realization $(C_e)_{e \in E}$, and all $n \ge N((C_e)_{e \in E})$, we have
$$ \forall e \in \pi_n, \quad \quad \sum_{g \le e}C_g^{-1}  \ge \exp \left(  \vert e \vert^{\gl}\right) $$
and thus
\begin{equation}
\liminf_{n \to \infty }\sum_{e \in \pi_n} \Psi_{RC}(e)^{1-\gep}=  \liminf_{n \to \infty } \sum_{e \in \pi_n} \left( \sum_{g \le e} C_g^{-1}  \right)^{-(1- \gep)}
\le \liminf_{n \to \infty } \sum_{e \in \pi_n} \exp \left(- (1-\gep) \vert e \vert^{\gl} \right)=0
\end{equation}
where the last equality is by $\gl > \ga$.
Therefore, we have $RT(\cT, \psi_{RC}) \le 1- \gep$ and conclude the proof by Proposition \ref{prop:CKSthm}.

\end{proof}

\subsection{Transience}
The goal of this subsection is the following proposition.
\begin{proposition} \label{prop:transience}
If $\gl< \ga=\IBN(\cT)$, then the RWRC associated with $\psi_{RC}$ is transient.
\end{proposition}

We adopt the definition of quasi-independence in \cite[Definition 4.1]{Collevecchio2019branching}.
\begin{definition} An edge percolation is quasi-independent if there exists a constant $C_Q \in (0, \infty)$ such that for any two edges $e_1,e_2 \in E$ with their common ancestor $e_1 \wedge e_2$ satisfy
\begin{equation}
\begin{aligned}
\bP \left( e_1, e_2 \in C(\varrho) \ \vert \ e_1\wedge e_2 \in C(\varrho)   \right) \le C_Q \cdot & \bP \left( e_1 \in C(\varrho) \ \vert \ e_1\wedge e_2 \in C(\varrho)   \right) \\
\cdot &\bP \left( e_2 \in C(\varrho) \ \vert \ e_1\wedge e_2 \in C(\varrho)   \right),
\end{aligned}
\end{equation}
where  $C(\varrho)$ is the connected component of the root  $\varrho$.
\end{definition}

\begin{proposition}[Proposition 4.2 in \cite{Collevecchio2019branching}]\label{prop:CHK4.2}
Consider an edge-percolation (not necessarily independent), such that
edges at generation $1$ are open almost surely, and for $e_1 \in E$ with $\vert e_1\vert>1$
\begin{equation}\label{percprob:openedge}
\bP \left( e_1 \in C(\varrho) \ \vert \ e_0 \in C(\varrho) \right)= \psi(e_1)>0
\end{equation}
where $e_0 \sim e_1$ and $e_0 < e_1$. If $RT(\cT, \psi)<1$, then $C(\varrho)$ is finite almost surely. If the percolation is quasi-independent and $RT(\cT, \psi)>1$, then $C(\varrho)$ is infinite with positive $\bP-$probability.
\end{proposition}

\begin{cor}\label{cor:subcritical}
 Let $\cT$ be a locally finite, infinite tree with $\alpha:=\IBN(\cT) \in (0,1)$. Fix a parameter $z \in (0, \infty)$ and perform a percolation on the edges (not necessarily independent) satisfying \eqref{percprob:openedge}. Assume there exists $n_0 \in \bbN$ and $c>0$ such that any $e\in E$ with $\vert e \vert >n_0$ satisfies
\begin{equation}\label{upbdasump:openprob}
\psi(e) \le 1-c \vert e \vert^{-z}.
\end{equation}
If $z \in (0, 1-\ga) $,  then the percolation is subcritical.
\end{cor}

\begin{proof}

For a cutset $\pi$, define $\vert \pi \vert \colonequals \inf \{\vert e \vert: \ e \in \pi \}.$ For $\gamma>0$, we have
$$ \inf_{\pi \in \Pi: \ \vert \pi \vert  \le n_0} \sum_{e \in \pi} \exp \left( - \vert e \vert^{\gamma} \right)  \ge \exp \left( -n_0^{\gamma} \right)>0.$$
Thus, for all $\gamma > \ga$,
$$ \inf_{\pi \in \Pi:\ \vert \pi \vert > n_0} \sum_{e \in \pi} \exp \left( - \vert e \vert^{\gamma} \right)=\inf_{\pi \in \Pi} \sum_{e \in \pi} \exp \left( - \vert e \vert^{\gamma} \right)=0.   $$

Moreover, for any $\gb>0$ and $z \in (0, 1- \ga)$, there exists a constant $C>0$ such that for all $e \in E$ with $\vert e\vert = n_0$
$$ \prod_{g \le e} \psi(g)  \le C \prod_{g \le e} \left( 1-c \vert  g\vert^{-z} \right). $$
We have
\begin{multline}
\inf_{\pi \in \Pi: \ \vert \pi \vert>n_0} \sum_{e \in \pi} \prod_{g \le e} \psi(e)^{\gb}
\le C^{\gb}  \inf_{\pi \in \Pi: \ \vert \pi \vert>n_0} \sum_{e \in \pi} \prod_{g \le e} \left( 1-c \vert  g\vert^{-z} \right)^{\gb} \\
\le \inf_{\pi \in \Pi: \ \vert \pi \vert>n_0} C^{\gb}  \sum_{e \in \pi} \exp \left( - c \gb\sum_{g \le e} \vert g \vert^{-z}  \right)
\le  \inf_{\pi \in \Pi: \ \vert \pi \vert>n_0} C' \sum_{e \in \pi} \exp \left( -\frac{c \gb}{1-z} \vert e \vert^{1-z} \right)=0
\end{multline}
where the first inequality uses \eqref{upbdasump:openprob}, the second inequality uses $1-x \le e^{-x}$ for all $x \ge 0$, and the last inequality is by integral approximation and $1-z >\ga$.
Therefore, $RT(\cT, \psi)<1$, and then we conclude the proof by Proposition \ref{prop:CHK4.2}.

\end{proof}

\begin{proposition}\label{prop:prepercolate}
Let $\cT$ be a locally finite, infinite tree with $\IBN(\cT)= \ga \in (0,1)$, and perform a quasi-independent percolation on the edges of $\cT$ such that \eqref{percprob:openedge} holds. Fix $z>1-\alpha$, and assume there exists $c>0$ and $n_0 \in \bbN$ such that for all $e$ with $ \vert e \vert > n_0$,
$$ \psi(e) \ge 1- c \vert e \vert^{-z}.$$
Then we have:
\begin{itemize}
\item[(1)] $RT(\cT, \psi)>1$ and $\cC(\varrho)$ is infinite with positive $\bbP-$probability.
\item[(2)] $\IBN(\cC(\varrho)) \ge 1-z$  with positive $\bbP-$probability.
\end{itemize}
\end{proposition}

\begin{proof}
We first deal with $(1)$. For $\pi\in \Pi$, let $\vert \pi \vert  \colonequals \inf \{ \vert  e\vert: \  e \in \pi \}$.
For any $\gb>1$, since $\psi(e)>0$ for all $e \in E$, we have
\begin{equation}
\inf_{\pi \in \Pi: \ \vert \pi \vert \le n_0} \sum_{e \in \pi} \prod_{g \le e} \psi(g)^{\gb}>0,
\end{equation}
and
\begin{multline}
\inf_{\pi \in \Pi: \ \vert \pi \vert>n_0} \sum_{e \in \pi} \prod_{g \le e} \psi(g)^{\gb}
\ge
c'\inf_{\pi \in \Pi: \ \vert \pi \vert>n_0} \sum_{e \in \pi} \prod_{g \le e} \left(1- c \vert g\vert^{-z} \right)^{\gb}
\\ \ge
c'\inf_{\pi \in \Pi: \ \vert \pi \vert>n_0} \sum_{e \in \pi} \exp \left( -2 c \gb \sum_{g \le e}  \vert g \vert^{-z} \right) \ge
c''\inf_{\pi \in \Pi: \ \vert \pi \vert>n_0} \sum_{e \in \pi} \exp \left( -\frac{  2c \gb}{1-z} \vert e \vert^{1-z} \right)>0
\end{multline}
where  the second inequality is by $1-x \ge e^{-2x}$ for all $x \in [0,1/2]$, and the last inequality is by $1-z<\ga$. Therefore, $RT(\cT, \psi)>1$. As the percolation is quasi-independent,  we conclude the proof for (1) by Proposition \ref{prop:CHK4.2}.\\

Now we move to (2). On the event $\{ \cC(\varrho) \text{ is infinite}\}$ which has positive $\bbP-$probability, we perform an independent percolation on the edges of $\cC(\varrho)$ on which an edge $e$ is open with probability $1-p \vert e\vert^{-z}$ independently. Let $\cC'(\varrho)$ denote the resulting cluster containing $\varrho$. Thus, it is a quasi-independent percolation on $\cT$ and
\begin{equation}
\bP \left( e \in \cC'(\varrho) \ \vert \ \cC'(\varrho)\right)
= \psi(e) (1- p\vert  e\vert^{-z}) \ge 1- (c+p) \vert e \vert^{-z}.
\end{equation}
We apply (1) to obtain that $\cC'(\varrho)$ is infinite with positive probability. By Corollary \ref{cor:subcritical}, we have
$\IBN(\cC(\varrho)) \ge 1- z$. Therefore, we conclude the proof.

\end{proof}

\subsection{Percolation linked to transience}\label{subsec:pertrans}
Given a realization of $(C_e)_{e \in E}$,    we define a percolation on the edges of $\cT$ as following:
$\{ e \text{ is open }\}$ holds almost surely for $\vert e \vert=1$, and for $\vert e\vert >1$
\begin{equation} \label{chooseE}
\left\{ e \text{ is open }\right\} \colonequals  \bigcap_{g \le e} \left\{ C_g^{-1} \le \exp \left(  \vert g \vert^{\gl}\right) \right\}.
\end{equation}
Moreover, we define
\begin{equation}\label{def:percomodel}
\psi_C(e) \colonequals
\begin{cases}
1 & \text { if } \vert e \vert=1,\\
\bbP \left( e \in \cC(\varrho) \ \vert \ e_0 \in \cC(\varrho) \right) & \text{ if } \vert e \vert>1,
\end{cases}
\end{equation}
where $e_0 \sim e$ and $e_0<e$.

\begin{proposition}\label{prop:percosupcritical}
The percolation defined in \eqref{def:percomodel} is quasi-independent. Moreover, for all $\gep \in(0, \ga/2)$ and $\gl=\ga-2\gep$, we have  $RT(\cT, \psi_C)>1$ and with positive $\bbP-$probability $\IBN(\cC(\varrho)) \ge \ga -\gep.$
\end{proposition}

\begin{proof}[Proof of  Proposition \ref{prop:transience} from Proposition \ref{prop:percosupcritical}]
Let $\cC(\varrho) \colonequals (V_C,E_C)$ denote the graph of the cluster containing $\varrho$. For any cutset $\pi \subset E$, $\pi \cap E_C$ is a cutset of $\cC(\varrho)$, and let $\Pi_C$ denote the set of cutsets of $\cC(\varrho)$. Thus for $\delta>0$, we have
\begin{equation}
\inf_{\pi \in \Pi} \sum_{e \in \pi} \left( \sum_{g \le e} C_g^{-1} \right)^{-(1+\delta)}
 \ge \inf_{\pi \in \Pi_C} \sum_{e \in \pi} \left( \sum_{g \le e} C_g^{-1} \right)^{-(1+\delta)}
  \ge
c \inf_{\pi \in \Pi_C} \sum_{e \in \pi} \exp \left( -(1+\delta) \vert e \vert^{\ga-\frac{3\gep}{2}} \right)>0
\end{equation}
where the second inequality is by \eqref{chooseE} so that
$$ \sum_{g \le e} C_g^{-1} \le \vert e \vert \exp \left( \vert e \vert^{\gl} \right) \le C_1  \exp \left( 2\vert e \vert^{\gl} \right) \le C  \exp \left( \vert e \vert^{\ga- \frac{3\gep}{2}} \right)$$
 and the last inequality is by Proposition \ref{prop:percosupcritical}.
Therefore $RT(\cT, \psi_{RC}) \ge 1+ \delta$, and then we conclude the proof by Proposition \ref{prop:CKSthm} and Kolmogorov's $0-1$ law.

\end{proof}

\begin{proof}[Proof of Proposition \ref{prop:percosupcritical}]
We first show that the percolation is quasi-independent. Let $e_1,e_2 \in E$ with $e=e_1 \wedge e_2$, and we have
\begin{multline}
\bbP \left( e_1,e_2 \in \cC(\varrho) \ \vert \ e \in  \cC(\varrho) \right)= \bbP \left( \bigcap_{i=1}^2 \bigcap_{e<g \le e_i} \left\{ C_g^{-1} \le \exp \left( \vert  g\vert^{\gl} \right) \right\}  \right)\\
=\bbP \left( \bigcap_{e<g \le e_1} \left\{ C_g^{-1} \le \exp \left( \vert  g\vert^{\gl} \right) \right\}  \right) \bbP \left( \bigcap_{e<g \le e_2} \left\{ C_g^{-1} \le \exp \left( \vert  g\vert^{\gl} \right) \right\}  \right)\\
=\bbP \left( e_1 \in \cC(\varrho) \ \vert \ e \in  \cC(\varrho) \right)\bbP \left( e_1 \in \cC(\varrho) \ \vert \ e \in  \cC(\varrho) \right).
\end{multline}
Therefore, the percolation is quasi-independent.\\
Moreover, for $e \sim e_0$ and $e_0< e$, we have
\begin{multline}
1- \psi_C(e)=\bbP \left( e \not\in  \cC(\varrho) \ \vert \ e_0 \in \cC(\varrho) \right)
= \bbP \left( C_e^{-1}> \exp \left( \vert e \vert^{\gl} \right) \right)\\
=L(\vert e\vert) \cdot \vert  e\vert^{-(1-\gl)} \le C \vert e\vert^{-(1-\gl)+\gep}= C \vert e\vert^{-(1-\ga+\gep)}
\end{multline}
where the second last inequality holds for all $\vert e \vert >n_0$ with some $n_0 \in \bbN$ as $L(n)=o(n^{\gep})$. Then we conclude the proof  by Proposition \ref{prop:prepercolate}. 
\end{proof}

Combining Proposition \ref{prop:recurRWRC} with Proposition \ref{prop:transience},  we conclude the proof of Theorem \ref{th:RWRCIBN}.

\section{Proof of Theorem \ref{th:fireintermediate}}
\label{sec:fireIBN}

In this section we prove Theorem \ref{th:fireintermediate} that the Intermediate Branching Number is the critical threshold for the firefighter problem on trees. Our strategy follows the strategy in \cite{Lehner2019firefighting} for the (standard) branching number on exponential growth trees. We first state two Lemmas proved by Lehner. We include their (short) proofs in our notation for sake of completeness.

Let $\varrho \in V$ be rooted, and $\vert v\vert$ is the graph distance between $\varrho$ and $v$. The following lemma says that it is sufficient to consider the following initial fire set
\begin{equation}
B(k) \colonequals \left\{v \in V: \ \vert v\vert \le k \right\}, \quad k \in \bbN.
\end{equation}

\begin{lemma}[Lemma 2.1 \cite{Lehner2019firefighting}]\label{lema:initialball}
 The graph $\cG$ satisfies $g$-containment if and only if for all $k \in \bbN$, there is a containable map $\Upsilon_k$ for $B(k)$.

\end{lemma}
\begin{proof}
Observe that if $Z_0' \subset Z_0$ and $\Upsilon$ is containable for $Z_0$, then $\Upsilon$ is also containable for $Z_0'$ since $Z_n' \subset Z_n$. Due to this monotonicity, we conclude the proof.
\end{proof}

When $\cG$ is a tree, we can go one step further due to the structure of trees.
Given $V_0 \subset V$, let $\cG-V_0$ be the (connected/disconnected) graph obtained by deleting all the edges whose further endpoints are in $V_0$.
 For a subset $U \subset V$, we say that $U$ is a surrounding set for $Z_0$ if $U \cap Z_0= \emptyset$, $U$ is a cutset and
 $Z_0$ is a subset of the finite component of $\cG-U$. Observe that a containable map $\Upsilon$ for $Z_0$ provides a surrounding set $U$ which is the set of those vertices with one neighbor in $Z_{\infty}=\cup_{n=0}^{\infty}Z_n$ and $U \cap Z_{\infty}=\emptyset$. Conversely, if there is a legal map $\Upsilon$ protect a sounding set $U$, then $\Upsilon$ is containable for $Z_0$. The following lemma gives a criterion to tell whether $\cG$ is $g$-containable.

\begin{lemma}[Lemma 2.2 \cite{Lehner2019firefighting}]\label{lema:ballcriterial}
 A locally finite, infinite tree $\cT$ is $g$-containable if and only if for each $B(k)$ with $k \in \bbN$ there exists a surrounding set $U$ for $B(k)$ and all $n \in \bbN$ such that
\begin{equation}\label{cond:surround}
\# U_{k+n}  \colonequals \# \left\{ v \in U: \ \vert v \vert  \le k+n \right\}  \le \sum_{i =1}^n g_i.
\end{equation}
\end{lemma}

\begin{proof} If $\cT$ is  $g$-containable, we take $Z_0=B(k)$ and the corresponding containable map $\Upsilon_k$ provides a sounding set $U$ for $B(k)$ by taking the protected vertices $U \colonequals \cup_{j=1}^{\infty}W_j$. From the definition of containable maps, we see that $U$ satisfies \eqref{cond:surround} for all $n \in \bbN$.

For each given $Z_0=B(k)$, now we assume that \eqref{cond:surround} is satisfied for all $n \in \bbN$.
We first order all the vertices in $U$ according to their distance to $\varrho$ such that vertices with smaller distance to $\varrho$ have higher priority (arbitrary order among vertices with the same distance to $\varrho$). The map $\Upsilon$ is defined as follows: $\Upsilon_k(1)$ is the set of the first $g_1$ highest priority vertices in $U$, and  $\Upsilon_k(n)$ is the set of the highest
$g_n$ priority vertices of $U \setminus \cup_{j=1}^{n-1} \Upsilon_k(j)$ for $n >1$. We can see that by \eqref{cond:surround}
 the map $\Upsilon_k$ is legal. Moreover, since $U$ is a surrounding set and protected by $\Upsilon$, then $\Upsilon$ is a containable map.
\end{proof}

\begin{proof}[Proof of Theorem \ref{th:fireintermediate} ]
We first show that $\gl_c \le \ga$. For any fixed $\gl \in (\ga,1)$, we define $g_n \colonequals \lfloor \exp \left( n^{\gl}\right) \rfloor$. For each fixed $B(k)$, we define
$$\gep=\exp \left( -k^{\gl}\right)-\exp \left( -(k+1)^{\gl} \right)>0. $$
Since $\gl>\ga$, by the definition of $\IBN$, there exists a cutset $\pi \subset E$ such that
$$ \sum_{e \in \pi} \exp \left( - \vert e\vert^{\gl} \right)< \gep. $$
Therefore, $\pi$ does not contain any edge in $B(k)$, \textit{i.e.} $\inf_{e \in \pi} \vert e \vert>k$. We define
\begin{equation}
U \colonequals \left\{v \in V: \ \exists e \in \pi, e^+=v \right\}
\end{equation}
which is a surrounding set for $B(k)$.  Recalling $U_n'=\{ v \in U: \ \vert v \vert=n\}$, we have
\begin{equation}\label{upbd:leveln}
\gep> \sum_{e \in \pi} \exp \left( -\vert e \vert^{\gl} \right)=\sum_{v \in U} \exp \left( -\vert v \vert^{\gl} \right)
\ge \sum_{v \in U_n'} \exp \left( -\vert v \vert^{\gl} \right)= \vert  U_n' \vert \exp \left(- n^{\gl} \right).
\end{equation}
Recalling $U_n= \{ v\in U: \ \vert v \vert \le n\}$, since $U_n= \emptyset$ for $n \le k$, by \eqref{upbd:leveln}, for $n \ge k+1$ we have
 $$ \vert U_n \vert =\sum_{j=k+1}^n \vert  U_j'\vert \le   \sum_{j=k+1}^n \left\lfloor \gep  \exp \left( j^{\gl} \right)  \right\rfloor  \le    \sum_{j=k+1}^n \left \lfloor \exp \left( (j-k)^{\gl} \right) \right\rfloor = \sum_{j=1}^{n-k} g_j$$
where we have used  $x^{\gl}+(1-x)^{\gl} \ge 1$ for all $x \in [0,1]$ in the  last inequality.
Since $k$ is arbitrary, we conclude the proof by Lemma \ref{lema:ballcriterial} and Lemma \ref{lema:initialball}.

\medskip

 Now we move to prove $\gl_c \ge \ga$. For any fixed $\gl \in (0, \ga)$, and any fixed $K>0$, we show that $\cT$ is not $g$-containable with $g_n \colonequals K \exp(n^{\gl})$.
 Taking $\gga \in (\gl, \ga)$, by the definition of $\IBN$ in \eqref{def:IBN}, we have
 \begin{equation}\label{def:gep}
 \gep \colonequals \inf_{\pi \in \Pi} \sum_{e \in \pi} \exp \left(- \vert e\vert^{\gga} \right)>0.
 \end{equation}
Moreover, we choose $k \in \bbN$ such that
\begin{equation}
 \sum_{n=k+1}^{\infty} n K \exp \left(-n^{\gga}+n^{\gl} \right)< \gep,
\end{equation}
which is possible as $\gl < \gga$.
For the chosen $k$ above, we claim that there is no surrounding set $U$ for $B(k)$ satisfying \eqref{cond:surround}. We argue by contradiction. Suppose there is one such $U$, and let
$\pi_U \colonequals \{ e \in E: \ e^+ \in U\}$. By \eqref{cond:surround} and $U_n'=\{v \in U: \ \vert v \vert=n \}$, we have
\begin{equation}\label{upbd:Un}
\#  U_n' \le \#  U_n  \le \sum_{j=1}^{n-k} g_j \le  n K \exp \left( n^{\gl} \right).
\end{equation}
As $\pi_U$ is a cutset, combining \eqref{def:gep} and \eqref{upbd:Un}, we have

\begin{multline}
\gep \le \sum_{e \in \pi_U} \exp \left(- \vert e\vert^{\gga} \right)= \sum_{v \in U} \exp \left(- \vert v\vert^{\gga} \right)= \sum_{n=k+1}^{\infty} \vert U_n' \vert \exp \left( - n^{\gga}\right)
\le \sum_{n=k+1}^{\infty} n K \exp\left( n^{\gl} \right)\exp \left( - n^{\gga}\right)< \gep
\end{multline}
which is a contradiction. By Lemma \ref{lema:ballcriterial}, we conclude the proof.

\end{proof}

\end{appendix}

\bibliographystyle{alpha}
\bibliography{library.bib}
\end{document}